\documentclass[11pt,a4paper]{article}
\usepackage{amsfonts}
\usepackage{amsmath}
\usepackage{amssymb}
\usepackage{amsthm}
\usepackage{hyperref}
\usepackage{mathtools}
\usepackage{graphicx}
\usepackage{color}
\usepackage[left=3.5cm,right=3.5cm, bottom=3.5cm, top=3.5cm]{geometry}
\usepackage[all]{xy}
\usepackage{tikz}
\usetikzlibrary{decorations.markings}
\usetikzlibrary{decorations.pathreplacing}
\usetikzlibrary{arrows,shapes,positioning}
\tikzstyle directed=[postaction={decorate,decoration={markings,    mark=at
position #1 with {\arrow{>}}}}] \tikzstyle
rdirected=[postaction={decorate,decoration={markings,   mark=at position #1 with
{\arrow{<}}}}]
\allowdisplaybreaks
\tikzset{anchorbase/.style={baseline={([yshift=-0.5ex]current bounding box.center)}},
  tinynodes/.style={font=\tiny,text height=0.75ex,text depth=0.15ex},
  smallnodes/.style={font=\scriptsize,text height=0.75ex,text depth=0.15ex},
}
\newcommand{\ru}{to [out=0,in=270]}
\newcommand{\rd}{to [out=0,in=90]}
\newcommand{\ur}{to [out=90,in=180]}

\newcommand{\lu}{to [out=180,in=270]}

\newcommand{\dr}{to [out=270,in=180]}
\newcommand{\dl}{to [out=270,in=0]}
\newcommand{\pu}{to [out=90,in=270]}
\newcommand{\pr}{to [out=0,in=180]}
\newcommand{\pd}{to [out=270,in=90]}

\newcommand\scalemath[2]{\scalebox{#1}{\mbox{\ensuremath{\displaystyle #2}}}}

\newcommand{\qbins}[2]{\genfrac[]{0pt}{1}{#1}{#2}}

\newcommand{\qb}[2]{\qbins{#1}{#2}}

\usepackage{bbm}
\def\C{{\mathbbm C}}
\def\N{{\mathbbm N}}

\def\Z{{\mathbbm Z}}
\def\Q{{\mathbbm Q}}

\def\QTf{{}_\bullet\mathrm{QT}_4}
\def\QT{\mathrm{QT}_4}
\def\T{\mathrm{T}_4}
\def\Tf{{}_\bullet\mathrm{T}_4}
\def\Fa{\mathcal{F}_a}
\def\Fi{\mathcal{F}_i}
\def\qc{qc}
\newcommand{\qp}[2]{(#1)_{#2}}
\newcommand{\qpp}[1]{(q^2)_{#1}}

\usepackage[T1]{fontenc}
\usepackage{selinput}  
\SelectInputMappings{  
   cacute={ć},          
   lslash={ł}           
}

\newtheorem{thm}{Theorem}[section]
\newtheorem{cor}[thm]{Corollary}

\newtheorem{lem}[thm]{Lemma}
\newtheorem{rem}[thm]{Remark}
\newtheorem{prop}[thm]{Proposition}
\newtheorem{defn}[thm]{Definition}

\title{Tangle addition and the knots-quivers correspondence}

\author{Marko Sto$\check{\text{s}}$i$\acute{\text{c}}$
  \and
  Paul Wedrich
 }

\newcommand{\Addresses}{{
  \bigskip
  \footnotesize

  Marko Sto$\check{\text{s}}$i$\acute{\text{c}}$
  
\noindent\textsc{CAMGSD, Departamento de Matem\'atica, Instituto Superior Tecnico, Av. Rovisco Pais, 1049-001 Lisbon, Portugal\\
\indent{\it and}\\
\noindent Mathematical Institute SANU, Knez Mihailova 36, 11000 Beograd, Serbia.}\par\nopagebreak
  \textit{E-mail address:} \texttt{mstosic@isr.ist.utl.pt}

  \medskip

  Paul Wedrich

\noindent \textsc{Max Planck Institute for Mathematics, Vivatsgasse 7, 53111 Bonn, Germany\\
\indent{\it and}\\
\noindent Mathematical Institute, University of Bonn,
Endenicher Allee 60, 53115 Bonn, Germany\\
\indent{\it and}\\
\noindent Mathematical Sciences Research Institute,
17 Gauss Way, Berkeley, CA 94720, USA.
}\par\nopagebreak
  \textit{E-mail address:} \texttt{p.wedrich@gmail.com}
  \par\nopagebreak
  \textit{Website:} \texttt{paul.wedrich.at}
}}

\date{}

\begin{document}
\maketitle

\begin{abstract}
    We prove that the generating functions for the one row/column colored
    HOMFLY-PT invariants of arborescent links are specializations of the
    generating functions of the motivic Donaldson-Thomas invariants of
    appropriate quivers that we naturally associate with these links. Our
    approach extends the previously established tangles-quivers correspondence
    for rational tangles to algebraic tangles by developing gluing formulas for
    HOMFLY-PT skein generating functions under Conway's tangle addition. As a consequence, we prove
    the conjectural links-quivers correspondence of
    Kucharski--Reineke--Sto$\check{\text{s}}$i$\acute{\text{c}}$--Su$\text{ł}$kowski
    for all arborescent links.
    
\end{abstract}

\section{Introduction}
The knots-quivers correspondence of
Kucharski--Reineke--Sto$\check{\text{s}}$i$\acute{\text{c}}$--
Su$\text{ł}$kowski \cite{KRSSlong,KRSSshort} proposes a relation between the
colored HOMFLY-PT polynomials of knots and the motivic Donaldson-Thomas
invariants of symmetric quivers. The main prediction is that the generating
function of the (anti-)symmetrically colored reduced HOMFLY-PT polynomials
$P_j(K)$ of any framed oriented knot $K$ can be expressed as

\begin{equation}
    \label{eq:knotqf}
    \sum_{j\geq 0} P_j(K)x^j = 
    \sum_{\textbf{d}=(d_1,\dots,d_m)\in \N^{m}} \hspace{-.5cm}(-1)^{R\cdot\textbf{d}} q^{S\cdot\textbf{d}} a^{A\cdot\textbf{d}} q^{\textbf{d}
    \cdot Q\cdot \textbf{d}^t} {d_1+\cdots+d_m \brack d_1,\dots,d_m} x^{d_1+\cdots+d_m}
    \end{equation}
    for some $m\geq 1$, $R,S,A\in \Z^m$ and a symmetric integer matrix $Q$ of
    size $m\times m$, where square brackets indicate quantum multinomial
    coefficients and $\N:=\Z_{\geq 0}$. Such expressions compactly encode the growth behavior of the
    colored HOMFLY-PT invariants (and also colored Jones polynomials), which
    is of central interest in quantum topology.
    \smallskip

Assuming the matrix $Q$ has non-negative entries\footnote{This can always be
achieved at the expense of a framing change on $K$.}, it can be interpreted as
the adjacency matrix of a quiver. As explained in \cite{KRSSlong}, the
right-hand side of \eqref{eq:knotqf} then appears as a specialisation of the
motivic DT-series of the quiver, i.e. the $\Z\times \N^m$-graded
Hilbert--Poincar\'{e} series of the corresponding cohomological Hall algebra
\cite{KS,E}. Moreover, the specialisation data is conjectured to be determined
by the Poincar\'{e} polynomial of the reduced triply-graded Khovanov--Rozansky
homology \cite{KR} of $K$.\smallskip

For geometric interpretations of the knots-quivers correspondence and relations
with colored HOMFLY-PT homology \cite{Wed3} we refer to \cite{EKL,PSS,PS}.
Possibly related connections between enumerative geometry and quantum knot
invariants have been studied in \cite{LM,OV,DHS,Mau,ORS,OS}. \smallskip

The knots-quivers conjecture has been verified for all knots with at most six
crossings and the infinite families of $(2,2n+1)$ torus knots and twist knots in
\cite{KRSSlong}. These verifications proceeded by ad hoc constructions of
generating function data as required by \eqref{eq:knotqf} for each such knot. \smallskip

In \cite{SW} we verified the knots-quivers conjecture and its natural extension
to a links-quivers conjecture for all 2-bridge links in a systematic way.
Motivated by the idea of \emph{skein theory with variable color}, we introduced
a way of encoding the generating functions of HOMFLY-PT invariants of $4$-ended
tangles in a ``quiver form'' analogous to \eqref{eq:knotqf}. We then showed how
to explicitly construct such quiver forms for all rational tangles
by induction on the crossing number, building on previous work in
\cite{Wed1,Wed2}. Finally, we verified that they provide generating function
data as required by \eqref{eq:knotqf} upon closing the rational tangle into a
2-bridge link.\smallskip

Furthermore, while expressions of the form \eqref{eq:knotqf} are certainly not
unique, not even for a fixed size $m\in \N$, the inductive description from
\cite{SW} assigns to each rational tangle, and thus also to its closure, a
distinguished quiver form expression.\smallskip

The purpose of this paper is to extend the relationship between $4$-ended
tangles and quivers beyond rational tangles and prove the links-quivers
conjecture for a larger class of links. 

\begin{thm}\label{thm:main} There exists a family $\QT$ of $4$-ended framed oriented tangles
with the following properties:
\begin{itemize}
    \item $\QT$ contains the trivial 2-strand tangle.
    \item $\QT$ is closed under diffeomorphisms of $(B^3, \partial B^3, \{4~pts\})$.
    \item $\QT$ is closed under Conway's tangle addition~\cite{Con}, the binary operation of
    gluing two $4$-ended tangles at pairs of boundary points as follows:
    \[\begin{tikzpicture} [scale=.5,anchorbase, tinynodes]
    \draw[thick] (-0.5,0) to (-0.5,1.5);
    \draw[thick] (0.5,0) to (0.5,1.5);
    \draw[thick, fill=white] (-.75,.5) rectangle (.75,1);
    \node at (0,.75) {$\tau_1$};
\end{tikzpicture}
\;,\;
 \begin{tikzpicture} [scale=.5,anchorbase, tinynodes]
    \draw[thick] (-0.5,0) to (-0.5,1.5);
    \draw[thick] (0.5,0) to (0.5,1.5);
    \draw[thick, fill=white] (-.75,.5) rectangle (.75,1);
    \node at (0,.75) {$\tau_2$};
    \end{tikzpicture}
    \qquad 
    \xrightarrow{}
    \qquad
    \begin{tikzpicture} [scale=.5,anchorbase, tinynodes]
        \draw[thick] (-0.5,0) to (-0.5,1.5);
        \draw[thick] (0.5,0.5) to (0.5,1) \ur (1,1.5) \rd (1.5,1);
        \draw[thick, fill=white] (-.75,.5) rectangle (.75,1);
        \draw[thick] (0.5,.5) \dr (1,0) \ru (1.5,0.5) to (1.5,1);
        \draw[thick] (2.5,0) to (2.5,1.5);
        \draw[thick, fill=white] (1.25,.5) rectangle (2.75,1);
        \node at (0,.75) {$\tau_1$};
        \node at (2,.75) {$\tau_2$};
        \end{tikzpicture}
    \]
    \item The appropriate analogue of the knots-quivers correspondence
    \eqref{eq:linkqf} holds for any link obtained by closing off a tangle in
    $\QT$.
\end{itemize} 
\end{thm}

The orbit of the trivial tangle under under diffeomorphisms of the 3-ball that
preserve the tangle boundary set-wise is called the set of rational tangles. Its
closure under tangle addition is called the set of algebraic tangles, which were
introduced by Conway in the context of link enumeration~\cite{Con}. Links
obtained as closures of \emph{algebraic tangles} are called \emph{arborescent},
or \emph{algebraic in the sense of Conway}\footnote{This is to distinguish them from
a different notion of algebraic links, which arise as links of plane curve
singularities.}.

\begin{cor} The links-quivers correspondence holds for all arborescent links and, in
particular, all Montesinos and pretzel links. 
\end{cor}
In particular, this provides an algorithm to compute one row/column colored
HOMFLY-PT invariants for these classes of links. For alternative approaches see
e.g.~\cite{Mir, GLL, Ano}. 

To summarise, our results in this paper show that the tangles-quivers
correspondence is robust under skein-theoretic operations, at least as far as
4-ended tangles and homogeneous one row/column colorings are concerned. This is
surprising, given that geometric and physical models for the large color limit
of HOMFLY-PT invariants usually work in the closed setting of knots and links.
However, in the 4-ended tangle case see \cite{Wed2} for elements of a geometric
model for categorified colored HOMFLY-PT invariants and \cite{KWZ1, KWZ2} for a
detailed study in the uncolored, categorified $\mathfrak{sl}_2$ case.

Extensions to tangles with more endpoints and different colors, and connections
to categorified invariants are interesting subjects for future work.

\subsection*{Structure of this paper}
The family $\QT$ is constructed in Section~\ref{sec:basics} and the gluing
result is proved in Section~\ref{sec:gluing}. In Section~\ref{sec:size} we show
that gluing is at worst bilinear in the size of the input quiver data.
Section~\ref{sec:examples} contains examples for pretzel tangles that suggest
that the minimal quiver data of a sum tangle is typically smaller than the
worst case estimate. 


\subsection*{Acknowledgements}
We would like to thank Eugene Gorsky, Mikhail Gorsky, Piotr Kucharski,
Anton Mellit, Ramadevi Pichai, Jacob Rasmussen, Piotr Su$\text{ł}$kowski and Claudius Zibrowius for useful
discussions, and Mikhail Gorsky and Piotr Su$\text{ł}$kowski for comments on a draft of this paper. 

This project started during the conference ``Categorification and Higher
Representation Theory'' at the Institute Mittag Leffler and we are grateful
to the organisers and the institute for their hospitality. P.W. would like to thank the
Departamento de Matem\'atica at the Instituto Superior T\'{e}cnico for supporting an
essential research visit in December 2019.

During the final preparation of this article we received the sad news of the
passing of John H. Conway. We take this opportunity to express our admiration
for the truly incomparable mathematician that he was.

\subsection*{Funding}

M.S. was partially supported by Portuguese {\it Funda\c c\~ao para a Ci\^encia e a Tecnologia} (FCT),
through the grant `Higher Structures and Applications', no. PTDC/MAT-PUR/31089/2017, and FCT Exploratory grant 
no. IF/00998/2015. M.S. was also partially supported by the Ministry of Education,
Science and Technological Development of the Republic of Serbia through Mathematical Institute SANU.
P.W. was partially supported by the Australian
Research Council grants `Braid groups and higher representation theory'
DP140103821 and `Low dimensional categories' DP160103479 while at the Australian
National University during early stages of this project. P.W. was also supported by the
National Science Foundation under Grant No. DMS-1440140, while in
residence at the Mathematical Sciences Research Institute in Berkeley,
California, during the Spring 2020 semester.

\section{HOMFLY-PT partition function}
\label{sec:basics}
Let $L$ be a link with $c$ components. Consider the reduced $\wedge^j$-colored HOMFLY-PT
invariant $P_j(L)$ for $j\geq 0$ and form the generating function
$P(L)=\sum_{j\geq 0} P_j(L)x^j$. 

\begin{defn} The HOMFLY-PT partition function $P(L)$ of a link with $c$ components is
said to be in quiver form, if it is presented as:
    \begin{equation}
        \label{eq:linkqf}
        P(L)=
        \sum_{\textbf{d}\in \N^{m}} (-q)^{S\cdot\textbf{d}} a^{A\cdot\textbf{d}} q^{\textbf{d}
        \cdot Q\cdot \textbf{d}^t} {|\textbf{d}| \brack \textbf{d}}  \qpp{|\textbf{d}|}^{1-c} x^{|\textbf{d}|}.
        \end{equation}
        for some $m\geq 1$, $S,A\in \Z^m$ and $Q$ a symmetric $m\times m$ integer matrix.
        Here $|\textbf{d}|$ denotes the sum of the entries of the
        vector $\textbf{d}$, $\qpp{k}=\prod_{i=1}^k(1-q^{2i})$ is a
        $q$-Pochhammer symbol, and
\small        
        \[\qb{d_1+\cdots +d_m}{d_1,\dots,d_m}:= 
        \frac{\qpp{d_1+\cdots +d_m}}{\qpp{d_1}\cdots\qpp{d_m}}\] 
       \normalsize
        is a quantum multinomial coefficient.
 \end{defn}

If $L$ is a knot and $P(L)$ is in quiver form, then the coefficients of $x^j$
are manifestly polynomial. Note, that \eqref{eq:linkqf} is slightly more rigid
than \eqref{eq:knotqf} since we require $R=S$. If $L$ is a link
with $c\geq 2$ components, then the coefficients exhibit the expected standard denominator
$\qpp{|\textbf{d}|}^{c-1}$ instead.

As explained in \cite[Section 4.3]{SW}, the generating function for the
symmetrically colored HOMFLY-PT invariants can be obtained from $P(L)$ by a
simple change of variables. 

\subsection{4-ended tangles and their HOMFLY-PT partition functions}
Let $\T$ denote the set of framed, oriented tangles with exactly $4$ boundary
points, which we represent by tangle diagrams as shown below, with boundary
points labeled by intercardinal directions. Let $\Tf$ be the subset of diagrams
for which the $SW$ boundary point is oriented inward. Within $\Tf$ we
distinguish three different configurations of boundary orientations, which we
call \emph{boundary types}:
\[
    \begin{tikzpicture} [scale=.5,anchorbase,tinynodes]
        \draw[thick] (-0.5,0) to (-0.5,1.5);
        \draw[thick] (0.5,0) to (0.5,1.5);
        \draw[thick, fill=white] (-.75,.5) rectangle (.75,1);
        \node at (-1.1,0.2) {$SW$};
        \node at (-1.1,1.3) {$NW$};
        \node at (1.1,0.2) {$SE$};
        \node at (1.1,1.3) {$NE$};
        \end{tikzpicture}
\qquad    
UP:\; \begin{tikzpicture} [scale=.5,anchorbase]
    \draw[thick, ->] (-0.5,0) to (-0.5,1.5);
    \draw[thick, ->] (0.5,0) to (0.5,1.5);
    \draw[thick, fill=white] (-.75,.5) rectangle (.75,1);
    \end{tikzpicture}
    \;,\quad 
    OP: \;\begin{tikzpicture} [scale=.5,anchorbase]
    \draw[thick, ->] (-0.5,0) to (-0.5,1.5);
    \draw[thick, <-] (0.5,0) to (0.5,1.5);
    \draw[thick, fill=white] (-.75,.5) rectangle (.75,1);
    \end{tikzpicture}
    \;,\quad RI:\; \begin{tikzpicture} [scale=.5,anchorbase] \draw[thick]
    (-0.5,0) to (-0.5,1.5); \draw[thick, <->] (0.5,0) to (0.5,1.5); \draw[thick,
    fill=white] (-.75,.5) rectangle (.75,1); \end{tikzpicture}\;.\] 
     We distinguish the types $UP$ and $RI$ even though they
    are related by a rotation.

\begin{defn} 
    \label{def:webs}   For any tangle $\tau\in \Tf$ and for $j\in \N$, we consider the
    $\wedge^j$-colored HOMFLY-PT invariant $\langle \tau \rangle_j$, an element
    of the free $\Q[a^{\pm 1}](q)$-module generated by the following basic webs
    for in the HOMFLY-PT skein theory (see \cite[Section 2.3]{SW} for a concise
    summary and \cite{MOY, CKM, TVW} for background):
    \[ UP[j,k] \;= \xy
    (0,0)*{
    \begin{tikzpicture} [scale=.75]
    \draw[thick, directed=.55] (-0.5,0) to (-0.5,0.5);
    \draw[thick, directed=.60] (-0.5,0.5) to (-0.5,1);
    \draw[thick, directed=.65] (-0.5,1) to (-0.5,1.5);
    \draw[thick, directed=.55] (0.5,0) to (0.5,0.5);
    \draw[thick, directed=.60] (0.5,0.5) to (0.5,1);
    \draw[thick, directed=.65] (0.5,1) to (0.5,1.5);
    \draw[thick, directed=.55] (0.5,0.4) to (-0.5,0.55);
    \draw[thick, directed=.55] (-0.5,0.95) to (0.5,1.1);
    \node at (-0.7,0.1) {\scriptsize $j$};
    \node at (0.7,0.1) {\scriptsize $j$};
    \node at (-0.7,1.4) {\scriptsize $j$};
    \node at (0.7,1.4) {\scriptsize $j$};
    \node at (0,0.25) {\scriptsize $k$};
    \node at (0,1.25) {\scriptsize $k$};
    \end{tikzpicture}
    }
    \endxy 
    \;,\quad
    OP[j,k] \;= \xy
    (0,0)*{
    \begin{tikzpicture} [scale=.75]
    \draw[thick, directed=.55] (-0.5,0) to (-0.5,0.5);
    \draw[thick, directed=.60] (-0.5,0.5) to (-0.5,1);
    \draw[thick, directed=.65] (-0.5,1) to (-0.5,1.5);
    \draw[thick, rdirected=.55] (0.5,0) to (0.5,0.5);
    \draw[thick, rdirected=.60] (0.5,0.5) to (0.5,1);
    \draw[thick, rdirected=.65] (0.5,1) to (0.5,1.5);
    \draw[thick, directed=.55] (-0.5,0.45) to [out=20,in=160] (0.5,0.45);
    \draw[thick, directed=.55] (0.5,1.05) to [out=200,in=340] (-0.5,1.05);
    \node at (-0.7,0.1) {\scriptsize $j$};
    \node at (0.7,0.1) {\scriptsize $j$};
    \node at (-0.7,1.4) {\scriptsize $j$};
    \node at (0.7,1.4) {\scriptsize $j$};
    \node at (0,1.25) {\scriptsize $k$};
    \node at (0,.25) {\scriptsize $k$};
    \end{tikzpicture}
    }
    \endxy
    \;,\quad
    RI[j,k]\;=
    \xy
    (0,0)*{
    \begin{tikzpicture} [scale=.75]
    \draw[thick, directed=.55] (-0.5,0) to (-0.5,0.5);
    \draw[thick, rdirected=.60] (-0.5,0.5) to (-0.5,1);
    \draw[thick, rdirected=.65] (-0.5,1) to (-0.5,1.5);
    \draw[thick, rdirected=.55] (0.5,0) to (0.5,0.5);
    \draw[thick, directed=.60] (0.5,0.5) to (0.5,1);
    \draw[thick, directed=.65] (0.5,1) to (0.5,1.5);
    \draw[thick, directed=.55] (-0.5,0.45) to [out=20,in=160] (0.5,0.45);
    \draw[thick, rdirected=.55] (0.5,1.05) to [out=200,in=340] (-0.5,1.05);
    \node at (-0.7,0.1) {\scriptsize $j$};
    \node at (0.7,0.1) {\scriptsize $j$};
    \node at (-0.7,1.4) {\scriptsize $j$};
    \node at (0.7,1.4) {\scriptsize $j$};
    \node at (0,1.25) {\scriptsize $k$};
    \end{tikzpicture}
    }
    \endxy
    ,\quad 0\leq k\leq j \] We will write $P(\tau)=\sum_{j\geq 0} \langle \tau
    \rangle_j$ for the generating function of these invariants and call it the
    HOMFLY-PT partition function.
\end{defn}

In \cite[Theorem 3.4]{SW} we proved that for a rational tangle $\tau\in \Tf$ of
boundary type $X$, the partition function $P(\tau)$ can be written in the form:
\begin{equation}
    \label{eq:ratgenf}
P(\tau)=
\sum_{\textbf{d}\in \N^{m+n}} (-q)^{S\cdot\textbf{d}} a^{A\cdot\textbf{d}} q^{\textbf{d}
\cdot Q\cdot \textbf{d}^t} {|\textbf{d}_a| \brack \textbf{d}_a}{|\textbf{d}_i| \brack \textbf{d}_i}  X[|\textbf{d}|,|\textbf{d}_a|]
\end{equation}
where $S,A\in \Z^{m+n}$ and $Q\in \Z^{(m+n)\times (m+n)}$ depend on the tangle.
The entries of the subvectors $\textbf{d}_a\in \N^m$ and $\textbf{d}_i\in \N^n$
of $\textbf{d}$ are called the active and inactive summation indices
respectively.

Expressions such as \eqref{eq:ratgenf} are suitable for studying the HOMFLY-PT
partition functions of rational tangles. However, for more general $4$-ended
tangles, in particular those with closed components, we need a more flexible
notion. 

\begin{defn}
    \label{def:qf}
If $\tau\in \Tf$ of boundary type $X$ with $c$ closed components,
then we say that $P(\tau)$ is in \emph{quiver form} if:
\begin{equation}
    \label{eq:qf}
    P(\tau)=
    \sum_{\textbf{d}\in \N^{m+n}} (-q)^{S\cdot\textbf{d}} a^{A\cdot\textbf{d}} q^{\textbf{d}
    \cdot Q\cdot \textbf{d}^t} {|\textbf{d}_a| \brack \textbf{d}_a}{|\textbf{d}_i| \brack \textbf{d}_i} \qpp{|\textbf{d}|}^{-c} X[|\textbf{d}|,|\textbf{d}_a|]
    \end{equation} 

We say that $P(\tau)$ is in \emph{active quiver form}, and write $\tau \in \Fa(X)$, if: 
\begin{equation}
    \label{eq:active}
P(\tau)=
\sum_{\textbf{d}\in \N^{m+n}} (-q)^{S\cdot\textbf{d}} a^{A\cdot\textbf{d}} q^{\textbf{d}
\cdot Q\cdot \textbf{d}^t} {|\textbf{d}_a| \brack \textbf{d}_a}{|\textbf{d}_i| \brack \textbf{d}_i} 
\frac{\qpp{|\textbf{d}_a|}}{\qpp{|\textbf{d}|}^{c}}  
X[|\textbf{d}|,|\textbf{d}_a|]
\end{equation}

We say that $P(\tau)$ is in \emph{inactive quiver form}, and write $\tau \in \Fi(X)$, if:
\begin{equation}
    P(\tau)=
    \sum_{\textbf{d}\in \N^{m+n}} (-q)^{S\cdot\textbf{d}} a^{A\cdot\textbf{d}} q^{\textbf{d}
    \cdot Q\cdot \textbf{d}^t} {|\textbf{d}_a| \brack \textbf{d}_a}{|\textbf{d}_i| \brack \textbf{d}_i} 
    \frac{\qpp{|\textbf{d}_i|}}{\qpp{|\textbf{d}|}^{c}}  
    X[|\textbf{d}|,|\textbf{d}_a|]
    \end{equation}
In all cases we ask that $m,n\in \N$ with $m+n\geq 1$, $S,A\in \Z^{m+n}$
and $Q\in \Z^{(m+n)\times (m+n)}$. Such data will be recorded as a triple, for
example in the case of \eqref{eq:active} by 
\[
    \left[
\scalemath{0.75}{X,\left( \begin{array}{c|c|c} 1-c & S_+ & A_+ \\ \hline\hline -c & S_- & A_- 
\end{array} \right)},
 \scalemath{0.75}{
 \left( 
\begin{array}{c|c} Q_{++} & Q_{+-} \\ \hline Q_{-+} & Q_{--} 
\end{array} \right)} \right],
\]
which indicates the boundary type, the presence of additional $q$-Pochhammer
symbols, as well as the vectors $S,A\in \Z^{m+n}$ and the matrix $Q\in
\Z^{(m+n)\times (m+n)}$ in block decomposition according to active ($+$) and
inactive ($-$) summation indices. 

\end{defn}

\begin{defn}
    Consider the following operations on $\T$.
\[
T\left(\begin{tikzpicture} [scale=.5,anchorbase]
\draw[thick] (-0.5,0) to (-0.5,1.5);
\draw[thick] (0.5,0) to (0.5,1.5);
\draw[thick, fill=white] (-.75,.5) rectangle (.75,1);
\end{tikzpicture}\right)
:=
\begin{tikzpicture} [scale=.5,anchorbase]
\draw[thick] (0.5,0) to (0.5,1) to [out=90,in=270] (-.5,2);
\draw[white, line width=.15cm] (-0.5,1) to [out=90,in=270] (.5,2);
\draw[thick] (-0.5,0) to (-0.5,1);
\draw[thick] (-0.5,1) to [out=90,in=270] (.5,2);
\draw[thick, fill=white] (-.75,.5) rectangle (.75,1);
\end{tikzpicture}
\quad, \quad
R\left(\begin{tikzpicture} [scale=.5,anchorbase]
\draw[thick] (-0.5,0) to (-0.5,1.5);
\draw[thick] (0.5,0) to (0.5,1.5);
\draw[thick, fill=white] (-.75,.5) rectangle (.75,1);
\end{tikzpicture}\right)
:=
\begin{tikzpicture} [scale=.5,anchorbase]
\draw[thick] (-0.5,0) to (-0.5,1.5);
\draw[thick] (0.5,.5) to (0.5,1) to [out=90,in=180](0.7,1.2) to [out=0,in=135] (1.5,.75) to [out=325,in=90] (2,0);
\draw[white, line width=.15cm] (0.7,.3) to [out=0,in=225] (1.5,.75) to [out=45,in=270] (2,1.5);
\draw[thick] (0.5,.5) to [out=270,in=180] (0.7,.3) to [out=0,in=225] (1.5,.75) to [out=45,in=270] (2,1.5);
\draw[thick, fill=white] (-.75,.5) rectangle (.75,1);
\end{tikzpicture}
\] 
Note that these operation preserve the orientation on the SW boundary point, and
thus restrict to endomorphisms of $\Tf$. We will also write $T^{-1}$ and
$R^{-1}$ for the inverse operations, which are given by gluing on the respective
inverse crossings.
\end{defn}

\begin{defn} We will consider six refined types of tangles $\tau\in \Tf$, which
encode boundary types and connectivity between boundary points:
    \begin{align*}UP_{par}:\; \begin{tikzpicture} [scale=.5,anchorbase]
        \draw[thick, ->] (-0.5,0) to (-0.5,1.5);
        \draw[thick, ->] (0.5,0) to (0.5,1.5);
        \draw[thick, fill=white] (-.75,.5) rectangle (.75,1);
        \draw[thin] (-0.5,.5) to (-0.5,1);
        \draw[thin] (0.5,.5) to (0.5,1);
        \end{tikzpicture}
        \;,\quad 
        &
        OP_{ud}: \;\begin{tikzpicture} [scale=.5,anchorbase]
        \draw[thick, ->] (-0.5,0) to (-0.5,1.5);
        \draw[thick, <-] (0.5,0) to (0.5,1.5);
        \draw[thick, fill=white] (-.75,.5) rectangle (.75,1);
        \draw[thin] (-0.5,.5) to (-0.5,1);
        \draw[thin] (0.5,.5) to (0.5,1);
        \end{tikzpicture}
        \;,\quad
        RI_{par}:\; \begin{tikzpicture} [scale=.5,anchorbase] \draw[thick]
        (-0.5,0) to (-0.5,1.5); \draw[thick, <->] (0.5,0) to (0.5,1.5); \draw[thick,
        fill=white] (-.75,.5) rectangle (.75,1); 
        \draw[thin] (-0.5,.5) \ur (0,.7) \rd (0.5,.5);
            \draw[thin] (-0.5,1) \dr (0,.8) \ru (0.5,1);
    \end{tikzpicture}\;
        ,\\
    UP_{cr}:\; \begin{tikzpicture} [scale=.5,anchorbase]
        \draw[thick, ->] (-0.5,0) to (-0.5,1.5);
        \draw[thick, ->] (0.5,0) to (0.5,1.5);
        \draw[thick, fill=white] (-.75,.5) rectangle (.75,1);
        \draw[thin] (0.5,.5) to (-0.5,1);
        \draw[thin] (-0.5,.5) to (0.5,1);
        \end{tikzpicture}
        \;,\quad 
        &
        OP_{lr}: \;\begin{tikzpicture} [scale=.5,anchorbase]
            \draw[thick, ->] (-0.5,0) to (-0.5,1.5);
            \draw[thick, <-] (0.5,0) to (0.5,1.5);
            \draw[thick, fill=white] (-.75,.5) rectangle (.75,1);
            \draw[thin] (-0.5,.5) \ur (0,.7) \rd (0.5,.5);
            \draw[thin] (-0.5,1) \dr (0,.8) \ru (0.5,1);
            \end{tikzpicture}
            \;,\quad\;\;
            RI_{cr}:\; \begin{tikzpicture} [scale=.5,anchorbase] \draw[thick]
                (-0.5,0) to (-0.5,1.5); \draw[thick, <->] (0.5,0) to (0.5,1.5); \draw[thick,
                fill=white] (-.75,.5) rectangle (.75,1); 
                \draw[thin] (-0.5,.5) \pu (0.5,1);
                    \draw[thin] (-0.5,1) \pd (0.5,.5);
            \end{tikzpicture}\;
.\end{align*}
For example, an $UP_{par}$ tangle has one strand directed from the $SW$ to the
$NW$ boundary point and the other strand directed from the $SE$ to the $NE$
boundary point, and possibly additional closed components.
\end{defn}

\subsection{The family of tangles $\QT$}
In this section we define a family of $4$-ended tangles $\QT\subset \T$ that
satisfy a \emph{tangles-quivers conjecture} that implies the links-quivers
conjecture for their closures. As before, we will focus on tangles with the inwards
orientation on the $SW$ boundary point.
\begin{defn}
    \label{def:TQC}
     Let $\QTf\subset \Tf$ denote the family of $4$-ended tangles $\tau$ that
satisfy the following conditions, depending on their type:
    \begin{enumerate}
        \item $\tau\in \Fa(UP)$ if $\tau$ is of type $UP_{par}$,
        \item $\tau\in \Fa(OP)$ if $\tau$ is of type $OP_{ud}$,
        \item $\tau\in \Fi(OP)$ if $\tau$ is of type $OP_{lr}$,
        \item $\tau\in \Fi(RI)$ if $\tau$ is of type $RI_{par}$,
        \item $C \tau\in \QTf$ for any $C\in \{T,T^{-1},R,R^{-1}\}$ if $\tau$ is
        of type $UP_{cr}$ or $RI_{cr}$.
    \end{enumerate}
    Let $\QT \subset \T$ be the family of those $4$-ended tangles $\tau$, for which
    there exists a rotation $r$ of the plane that takes $\tau$ to $r(\tau)\in \QTf$.
\end{defn}

We have the following symmetries.

\begin{lem}
    \label{lem:rotatetwo} If $\tau \in \QTf$, then so is 
    \begin{enumerate}
        \item the mirror image $-\tau\in \QTf$, 
\item the reflection across the $SW$-$NE$-diagonal $r_d(\tau)\in \QTf$,
\item the $\pi$-rotation around the vertical axis $r_{v}(\tau)\in \QTf$ for type $UP$,
\item the $\pi$-rotation around the horizontal axis $r_{h}(\tau)\in \QTf$ for type $RI$.
    \end{enumerate} 
\end{lem}
\begin{proof}
Mirroring acts by inverting the variable $q$, which preserves the existence of
active resp. inactive quiver forms as required. The reflection interchanges
$OP[j,k]$ with $OP[j,j-k]$ and $UP[j,k]$ with $RI[j,j-k]$, and thus also active
with inactive variables. The $\pi$-rotations simply preserve the generating
functions, because the corresponding basis webs are invariant by virtue of the
``square switch'' web relation. Thus all symmetries considered here preserve the
conditions in Definition~\ref{def:TQC}.
\end{proof}

The choice of the rotation $r$ in
the Definition~\ref{def:TQC} of $\QT$ is immaterial. 
\begin{lem} Let $r$ be a rotation in the plane that takes $\tau \in \Tf$ to
$r(\tau)\in \Tf$. Then we have
    \[\tau \in \QTf \iff r(\tau)\in \QTf.\]
\end{lem}
\begin{proof} The only relevant rotations are the $\pi/2$-rotation interchanging
$UP[j,k]$ and $RI[j,j-k]$ and swapping active and inactive summation indices,
and the $\pi$-rotation which preserves generating functions of type $OP$. Using
Lemma~\ref{lem:rotatetwo} it is straightforward to check that both rotations
preserve the conditions in Definition~\ref{def:TQC}.
\end{proof}

The next lemma was a key result in \cite{SW}.

\begin{lem}
    \label{lem:doublecrossingoperations}
    We have the following:
    \begin{itemize}
        \item if $\tau\in \Fa(UP)$, then $T^{\pm 2}\tau\in \Fa(UP)$, $RT\tau\in \Fi(OP)$ and $R^{\pm 1}\tau\in \Fa(UP)$,
        \item if $\tau\in \Fa(OP)$, then $T^{\pm 2}\tau\in \Fa(OP)$, $RT\tau\in \Fi(RI)$ and $R^{\pm 1}\tau\in \Fa(OP)$,
        \item if $\tau\in \Fi(OP)$, then $R^{\pm 2}\tau\in \Fi(OP)$, $TR\tau\in \Fa(UP)$ and $T^{\pm 1}\tau\in \Fi(OP)$,
        \item if $\tau\in \Fi(RI)$, then $R^{\pm 2}\tau\in \Fi(RI)$, $TR\tau\in \Fa(OP)$ and $T^{\pm 1}\tau\in \Fi(RI)$.
    \end{itemize}
\end{lem}
\begin{proof} For positive crossing operations, these properties are proven in Lemma 4.5,
4.6 and 4.7 of \cite{SW}. The proof for negative crossing operations is analogous.
\end{proof}

\begin{lem}
    \label{lem:cr-equic-cond}
    If $\tau$ is of type $UP_{cr}$ or $RI_{cr}$, then we have:
    \begin{align*}
        \tau \in \QTf &\iff T\tau \in \QTf \iff T^{-1}\tau \in \QTf \\
        &\iff R\tau \in \QTf \iff R^{-1}\tau \in \QTf. 
\end{align*}
    In particular, the four conditions in item 5 of Definition~\ref{def:TQC} are equivalent.
\end{lem}
\begin{proof} For the sake of concreteness, suppose $\tau$ is of type $UP_{cr}$,
and so $T^{\pm 1}\tau$ is of type $UP_{par}$ and $R^{\pm 1}\tau$ is of type
$OP_{lr}$. Then Lemma~\ref{lem:doublecrossingoperations} provides the following
equivalences and implications:
\begin{align*}
    T\tau \in \Fa(UP) 
    &\overset{T^{\pm 2}}{\iff} T^{-1}\tau \in \Fa(UP) \\
    &\overset{RT}{\implies} R\tau \in \Fi(OP) \\
    &\overset{R^{\pm 2}}{\iff} R^{-1}\tau \in \Fi(OP) \overset{TR}{\implies} T\tau \in \Fa(UP) 
\end{align*}
Any of these conditions is equivalent to all the others, and thus also to $\tau
\in \QTf$.
\end{proof}

\begin{prop}\label{prop:TCQcross} We have $T^{\pm 1}\QTf\subset \QTf$ and $R^{\pm 1}\QTf\subset \QTf$. 
\end{prop}
\begin{proof} This follows from Lemma~\ref{lem:doublecrossingoperations} and
Lemma~\ref{lem:cr-equic-cond}.
\end{proof}

\begin{lem}\label{lem:QTcross} $\QT$ is closed under attaching arbitrary
crossings between pairs of neighbouring boundary points.
\end{lem}
\begin{proof}
    Let $\tau \in \QT$. Suppose the crossing is to be attached between two
    boundary points which are not both incoming. Instead of attaching the
    crossing, we first rotate the free incoming boundary point into the $SW$
    position and obtain a tangle $r(\tau)\in \QTf$. Then we attach the crossing,
    which results in a tangle $\tau'\in \QTf$ by
    Proposition~\ref{prop:TCQcross}. After rotating back, the desired tangle is
    $r^{-1}(\tau')\in \QT$.  
    
    If the crossing is to be attached between the two incoming boundary points,
    we again rotate $\tau$ until we get an $UP$-tangle $r(\tau)\in \QTf$. Then
    by Lemma~\ref{lem:rotatetwo} attaching the crossing on the bottom produces a
    tangle in $\QTf$ if and only if attaching the crossing at the top does, but
    the latter is again covered by Proposition~\ref{prop:TCQcross} and we
    proceed as before. 
\end{proof}

\begin{thm}
    \label{thm:QTdiff} $\QT$ is closed under diffeomorphisms of the 3-ball, which fix the
boundary set-wise.
\end{thm}
\begin{proof} $\QT$ is invariant under isotopies relative to the boundary by
skein theory. Further, the corresponding mapping class group is generated by
twists on the boundary, which are covered by Lemma~\ref{lem:QTcross}, and
mirroring, which is covered by Lemma~\ref{lem:rotatetwo}.
\end{proof}

\begin{thm} $\QT$ contains all rational tangles.
\end{thm}
\begin{proof} The trivial tangle is in $\QT$ and any rational tangle is built
from it by successively attaching crossings between pairs of neighbouring
boundary points.
\end{proof}

In Section~\ref{sec:gluing} we will prove that $\QT$ contains all algebraic
4-ended tangles.

\subsection{From tangle quivers to link quivers}
If a framed, oriented link $L$ is obtained from a tangle $\tau\in \T$ by
connecting the northern boundary points to the southern boundary points (resp.
eastern to western) by arcs in the plane, then we write $L=Cl_{NS}(\tau)$ (resp.
$L=Cl_{EW}(\tau)$) and say that $L$ is a closure of $\tau$.

\begin{lem}\label{lem:closure} Let $L$ be a framed, oriented link, obtained as
    the closure of $\tau\in \QT$, then $P(L)$ can be
    written in quiver form \eqref{eq:linkqf}. 
    \end{lem}
    \begin{proof} 
        Up to isotopy, we may assume $L=Cl_{NS}(\tau)$ for a tangle $\tau\in
        \QTf$ of type $UP$ with $c$ closed components. This may require a change
        of framing, which, however, does not affect the existence of a quiver
        form for $P(L)$, see \cite[Section 4.3]{SW}. 

        If $\tau$ is of type $UP_{par}$, then $L$ has $c+2$ components. Applying
        the closure rule from \cite[Section 2.3.3]{SW} to $P(\tau)$ directly
        produces an expression for $P(L)$ in quiver form.  
       
        If $\tau$ is of type $UP_{cr}$, then $L$ has only $c+1$ components and
        more care is needed. In this case  $T^{-1}\tau$ is of type $UP_{par}$
        and in $\QTf$, and thus $P(T^{-1}\tau)$ may be assumed to be in active
        quiver form. Now the closure-after-top-crossing rule from \cite[Lemma
        4.8]{SW} implies that $P(L)$ can be written in quiver form. 
    \end{proof}

\section{Adding 4-ended tangles}
\label{sec:gluing}
We will now consider the binary addition operation on $4$-ended tangles, which
is given by gluing along pairs of adjacent boundary points, provided the
orientations are compatible there. 
\[+ \colon
    \left( \begin{tikzpicture} [scale=.5,anchorbase, tinynodes]
    \draw[thick] (-0.5,0) to (-0.5,1.5);
    \draw[thick] (0.5,0) to (0.5,1.5);
    \draw[thick, fill=white] (-.75,.5) rectangle (.75,1);
    \node at (0,.75) {$\tau_1$};
\end{tikzpicture}
\;,\;
 \begin{tikzpicture} [scale=.5,anchorbase, tinynodes]
    \draw[thick] (-0.5,0) to (-0.5,1.5);
    \draw[thick] (0.5,0) to (0.5,1.5);
    \draw[thick, fill=white] (-.75,.5) rectangle (.75,1);
    \node at (0,.75) {$\tau_2$};
    \end{tikzpicture}
    \right)
    \quad 
    \xrightarrow{}
    \quad
    \begin{tikzpicture} [scale=.5,anchorbase, tinynodes]
        \draw[thick] (-0.5,0) to (-0.5,1.5);
        \draw[thick] (0.5,0.5) to (0.5,1) \ur (1,1.5) \rd (1.5,1);
        \draw[thick, fill=white] (-.75,.5) rectangle (.75,1);
        \draw[thick] (0.5,.5) \dr (1,0) \ru (1.5,0.5) to (1.5,1);
        \draw[thick] (2.5,0) to (2.5,1.5);
        \draw[thick, fill=white] (1.25,.5) rectangle (2.75,1);
        \node at (0,.75) {$\tau_1$};
        \node at (2,.75) {$\tau_2$};
        \end{tikzpicture}
    \]
    Our goal for this section is to prove the following theorem.

    \begin{thm}\label{thm:mainthmbig} Let $\tau_1,\tau_2\in \QT$ with
    orientations such that $\tau_1+\tau_2$ is defined. Then  $\tau_1+\tau_2\in
    \QT$.
            \end{thm}
    
    For the purpose of deciding whether a glued tangle $\tau_1+\tau_2$ belongs
    to $\QT$ we may assume $\tau_1+\tau_2\in \Tf$. (Otherwise we would move it
    in such a position using a diffeomorphism of the 3-ball and
    Theorem~\ref{thm:QTdiff}.) With the orientation of the SW boundary point
    fixed as incoming, there exist only five orientation configurations:
\[
    \begin{tikzpicture} [scale=.5,anchorbase, tinynodes]
        \draw[thick,directed=.25,rdirected=.85] (-0.5,0) to (-0.5,1.5);
        \draw[thick,directed=.55] (0.5,1) \ur (1,1.5) \rd (1.5,1);
        \draw[thick] (0.5,0.5) to (0.5,1);
        \draw[thick, fill=white] (-.75,.5) rectangle (.75,1);
        \draw[thick,directed=.55] (0.5,.5) \dr (1,0) \ru (1.5,0.5);
        \draw[thick] (1.5,0.5) to (1.5,1);
        \draw[thick,rdirected=.2,directed=.9] (2.5,0) to (2.5,1.5);
        \draw[thick, fill=white] (1.25,.5) rectangle (2.75,1);
        \node at (0,.7) {$RI$};
        \node at (2,.7) {$RI$};
    \end{tikzpicture}\;,\quad
    \begin{tikzpicture} [scale=.5,anchorbase, tinynodes]
        \draw[thick,directed=.25,directed=.9] (-0.5,0) to (-0.5,1.5);
        \draw[thick,rdirected=.55] (0.5,1) \ur (1,1.5) \rd (1.5,1);
        \draw[thick] (0.5,0.5) to (0.5,1);
        \draw[thick, fill=white] (-.75,.5) rectangle (.75,1);
        \draw[thick,directed=.55] (0.5,.5) \dr (1,0) \ru (1.5,0.5);
        \draw[thick] (1.5,0.5) to (1.5,1);
        \draw[thick,directed=.25,directed=.9] (2.5,0) to (2.5,1.5);
        \draw[thick, fill=white] (1.25,.5) rectangle (2.75,1);
        \node at (0,.7) {$OP$};
        \node at (2,.7) {$UP$};
    \end{tikzpicture}\;,\quad
    \begin{tikzpicture} [scale=.5,anchorbase, tinynodes]
        \draw[thick,directed=.25,directed=.9] (-0.5,0) to (-0.5,1.5);
        \draw[thick,rdirected=.55] (0.5,1) \ur (1,1.5) \rd (1.5,1);
        \draw[thick] (0.5,0.5) to (0.5,1);
        \draw[thick, fill=white] (-.75,.5) rectangle (.75,1);
        \draw[thick,directed=.55] (0.5,.5) \dr (1,0) \ru (1.5,0.5);
        \draw[thick] (1.5,0.5) to (1.5,1);
        \draw[thick,rdirected=.2,rdirected=.85] (2.5,0) to (2.5,1.5);
        \draw[thick, fill=white] (1.25,.5) rectangle (2.75,1);
        \node at (0,.7) {$OP$};
        \node at (2,.7) {$OP$};
    \end{tikzpicture}\;,\quad
    \begin{tikzpicture} [scale=.5,anchorbase, tinynodes]
        \draw[thick,directed=.25,directed=.9] (-0.5,0) to (-0.5,1.5);
        \draw[thick,directed=.55] (0.5,1) \ur (1,1.5) \rd (1.5,1);
        \draw[thick] (0.5,0.5) to (0.5,1);
        \draw[thick, fill=white] (-.75,.5) rectangle (.75,1);
        \draw[thick,rdirected=.55] (0.5,.5) \dr (1,0) \ru (1.5,0.5);
        \draw[thick] (1.5,0.5) to (1.5,1);
        \draw[thick,directed=.25,directed=.9] (2.5,0) to (2.5,1.5);
        \draw[thick, fill=white] (1.25,.5) rectangle (2.75,1);
        \node at (0,.7) {$\scalebox{-1}[1]{UP}$};
        \node at (2,.7) {$\scalebox{-1}[1]{OP}$};
    \end{tikzpicture}\;,\quad
    \begin{tikzpicture} [scale=.5,anchorbase, tinynodes]
        \draw[thick,directed=.25,directed=.9] (-0.5,0) to (-0.5,1.5);
        \draw[thick,directed=.55] (0.5,1) \ur (1,1.5) \rd (1.5,1);
        \draw[thick] (0.5,0.5) to (0.5,1);
        \draw[thick, fill=white] (-.75,.5) rectangle (.75,1);
        \draw[thick,rdirected=.55] (0.5,.5) \dr (1,0) \ru (1.5,0.5);
        \draw[thick] (1.5,0.5) to (1.5,1);
        \draw[thick,rdirected=.2,rdirected=.85] (2.5,0) to (2.5,1.5);
        \draw[thick, fill=white] (1.25,.5) rectangle (2.75,1);
        \node at (0,.7) {$\scalebox{1}[1]{UP}$};
        \node at (2,1) {$\scalebox{1}[-1]{UP}$};
    \end{tikzpicture}
    \]
We will only consider the first and the second configuration. The fourth
configuration is a rotated version of the second, and the fifth is related to the third by a Reidemeister 2 move.
\begin{gather*}
    \begin{tikzpicture} [scale=.5,anchorbase, tinynodes]
        \draw[thick,directed=.25,directed=.9] (-0.5,0) to (-0.5,1.5);
        \draw[thick,directed=.55] (0.5,1) \ur (1,1.5) \rd (1.5,1);
        \draw[thick] (0.5,0.5) to (0.5,1);
        \draw[thick, fill=white] (-.75,.5) rectangle (.75,1);
        \draw[thick,rdirected=.55] (0.5,.5) \dr (1,0) \ru (1.5,0.5);
        \draw[thick] (1.5,0.5) to (1.5,1);
        \draw[thick,rdirected=.2,rdirected=.85] (2.5,0) to (2.5,1.5);
        \draw[thick, fill=white] (1.25,.5) rectangle (2.75,1);
        \node at (0,.7) {$\scalebox{1}[1]{UP}$};
        \node at (2,1) {$\scalebox{1}[-1]{UP}$};
    \end{tikzpicture}
    \;=\;
    \begin{tikzpicture} [scale=.5,anchorbase, tinynodes]
        \draw[thick,directed=.25,directed=.9] (-0.5,0) to (-0.5,1.5);
        \draw[thick,directed=.25] (0.5,1) \ur (.75,1.25) \pr (1.5,.5) \pr (2.25,1.25) \rd (2.5,1);
        \draw[thick] (0.5,0.5) to (0.5,1);
        \draw[thick, fill=white] (-.75,.5) rectangle (.75,1);
        \draw[white, line width=.15cm](.75,0.25) \pr (1.5,1) \pr (2.25,0.25) ;
        \draw[thick,rdirected=.25] (0.5,.5) \dr (.75,0.25) \pr (1.5,1) \pr (2.25,0.25)  \ru (2.5,0.5);
        \draw[thick] (2.5,0.5) to (2.5,1);
        \draw[red, dashed] (1.5,0) to (1.5,1.5);
        \draw[thick,rdirected=.2,rdirected=.85] (3.5,0) to (3.5,1.5);
        \draw[thick, fill=white] (2.25,.5) rectangle (3.75,1);
        \node at (0,.7) {$\scalebox{1}[1]{UP}$};
        \node at (3,1) {$\scalebox{1}[-1]{UP}$};
    \end{tikzpicture}
    \;\in\; 
    \begin{tikzpicture} [scale=.5,anchorbase, tinynodes]
        \draw[thick,directed=.25,directed=.9] (-0.5,0) to (-0.5,1.5);
        \draw[thick,rdirected=.55] (0.5,1) \ur (1,1.5) \rd (1.5,1);
        \draw[thick] (0.5,0.5) to (0.5,1);
        \draw[thick, fill=white] (-.75,.5) rectangle (.75,1);
        \draw[thick,directed=.55] (0.5,.5) \dr (1,0) \ru (1.5,0.5);
        \draw[thick] (1.5,0.5) to (1.5,1);
        \draw[thick,rdirected=.2,rdirected=.85] (2.5,0) to (2.5,1.5);
        \draw[thick, fill=white] (1.25,.5) rectangle (2.75,1);
        \node at (0,.7) {$OP$};
        \node at (2,.7) {$OP$};
    \end{tikzpicture}
\end{gather*}
Finally, the third configuration can be reduced to the second at the expense of twisting the boundary points.
\begin{gather*}
R\left(
    \begin{tikzpicture} [scale=.5,anchorbase, tinynodes]
    \draw[thick,directed=.25,directed=.9] (-0.5,0) to (-0.5,1.5);
    \draw[thick,rdirected=.55] (0.5,1) \ur (1,1.5) \rd (1.5,1);
    \draw[thick] (0.5,0.5) to (0.5,1);
    \draw[thick, fill=white] (-.75,.5) rectangle (.75,1);
    \draw[thick,directed=.55] (0.5,.5) \dr (1,0) \ru (1.5,0.5);
    \draw[thick] (1.5,0.5) to (1.5,1);
    \draw[thick,rdirected=.2,rdirected=.85] (2.5,0) to (2.5,1.5);
    \draw[thick, fill=white] (1.25,.5) rectangle (2.75,1);
    \node at (0,.7) {$OP$};
    \node at (2,.7) {$OP$};
\end{tikzpicture}
\right)
\;=\;
\begin{tikzpicture} [scale=.5,anchorbase, tinynodes]
    \draw[thick,directed=.25,directed=.9] (-0.5,0) to (-0.5,1.5);
    \draw[thick,rdirected=.55] (0.5,1) \ur (1,1.5) \rd (1.5,1);
    \draw[thick] (0.5,0.5) to (0.5,1);
    \draw[thick, fill=white] (-.75,.5) rectangle (.75,1);
    \draw[thick,directed=.55] (0.5,.5) \dr (1,0) \ru (1.5,0.5);
    \draw[thick] (1.5,0.5) to (1.5,1);
    \draw[thick,rdirected=.9] (2.5,.5) to (2.5,1) to [out=90,in=180](2.7,1.2) to [out=0,in=135] (3.5,.75) to [out=325,in=90] (4,0);
\draw[white, line width=.15cm] (2.7,.3) to [out=0,in=225] (3.5,.75) to [out=45,in=270] (4,1.5);
\draw[thick, directed=.9] (2.5,.5) to [out=270,in=180] (2.7,.3) to [out=0,in=225] (3.5,.75) to [out=45,in=270] (4,1.5);
    \draw[thick, fill=white] (1.25,.5) rectangle (2.75,1);
    \node at (0,.7) {$OP$};
    \node at (2,.7) {$OP$};
\end{tikzpicture}
\;\in \;
\begin{tikzpicture} [scale=.5,anchorbase, tinynodes]
    \draw[thick,directed=.25,directed=.9] (-0.5,0) to (-0.5,1.5);
    \draw[thick,rdirected=.55] (0.5,1) \ur (1,1.5) \rd (1.5,1);
    \draw[thick] (0.5,0.5) to (0.5,1);
    \draw[thick, fill=white] (-.75,.5) rectangle (.75,1);
    \draw[thick,directed=.55] (0.5,.5) \dr (1,0) \ru (1.5,0.5);
    \draw[thick] (1.5,0.5) to (1.5,1);
    \draw[thick,directed=.25,directed=.9] (2.5,0) to (2.5,1.5);
    \draw[thick, fill=white] (1.25,.5) rectangle (2.75,1);
    \node at (0,.7) {$OP$};
    \node at (2,.7) {$UP$};
\end{tikzpicture}
\end{gather*}

Theorem~\ref{thm:mainthmbig} thus follows from the following.

\begin{thm}\label{thm:mainthm} Suppose that $\tau_1,\tau_2\in \QTf$ are either
both of type $RI$, or $\tau_1$ is of type $OP$ and $\tau_2$ of type $UP$. Then we
have $\tau_1+\tau_2\in \QTf$.
    \end{thm}

The proof of this theorem will occupy the remainder of this section. For given
$\tau_1,\tau_2\in \QTf$, we know that $P(\tau_1)$ and $P(\tau_2)$ (or the
partition functions for tangles that differ by the twisting of two boundary
points) admit very special expressions, depending on the type and connectivity
of the tangle. Our task will be to show that the same holds for the glued tangle
$\tau_1+\tau_2$. 
The first step is to
compute how the skein basis elements glue.

    \subsection{Gluing of basis webs}
This section details how the horizontal gluings of basic webs from
Definition~\ref{def:webs} expand into linear combinations of basic webs. The
proofs of the lemmas here are straightforward diagrammatic computations using
the local relation in the HOMFLY-PT skein theory (see e.g. \cite[Fig. 1]{SW}) and
therefore omitted.

In the following we will encounter certain linear expressions in the
variables $k$, $l$, and $j$, which we indicate by the generic expression $\#$ if their precise
form is immaterial. Similarly we write $\#\#$ for quadratic expressions
in these variables. 

    \begin{lem}
        \label{lem:RIRI}
        The result of gluing the skein basis elements $RI[j,k]$ and $RI[j,l]$ can be
        expressed in terms of $RI$ basis elements as follows:
        \begin{gather*}
            RI[j,k] + RI[j,l] = 
         \sum_{0\leq t}  M \qb{k'+l'}{k' , l'} \qb{k'+l'-t}{t, k'-t, l'-t} RI[j,k+l-j+t]
        \end{gather*}
        where we write $k^\prime=j-k$ and $l^\prime=j-l$, and $M$ is of the form $q^{\#\#}$. Note that the non-zero
        summands occur only for $0\leq t\leq \min(k',l')$. 
    \end{lem}

\begin{lem}
    \label{lem:OPUP}
    The result of gluing the skein basis elements $OP[j,k]$ and $UP[j,l]$ can be
    expressed in terms of $UP$ basis elements as follows:
    \begin{gather*}
    OP[j,k] + UP[j,l] =
    M
    \frac{\qp{a^2q^{\#}}{k^\prime}}{\qpp{k^\prime}}
    \sum_{0\leq t}   \qb{l^\prime+t}{t, k^\prime-t , l^\prime-k^\prime+t  } UP[j,l-t]
    \end{gather*}
    where we again write $k^\prime=j-k$ and $l^\prime=j-l$ and $M$ is of the
    form $a^{\#}q^{\#\#}$. Note that the non-zero summands occur only for
    $\max(0,k'-l')\leq t\leq \min(l,k')$.
\end{lem}

In an intermediate step, we will also need a lemma for gluings of type $OP+ OP$.

\begin{lem}
    \label{lem:OPOP}
    The result of gluing the skein basis elements $OP[j,k]$ and $OP[j,l]$ can be
        expressed in terms of $OP$ basis elements as follows:
    \begin{align*}
        OP[j,k] + OP[j,l] 
        &= \sum_{0\leq t} M  \qb{l'+k'-t}{k',l'-t} \qb{l'+k'-t}{l',k'-t}   
        \frac{\qp{a^2q^{\#}}{t}}{\qpp{t}} OP[j,k+l+t-j]\\
        &=  \sum_{0\leq t} M  \qb{l'+k'-t}{t,k'-t,l'-t}    
        \frac{\qpp{k'+l'-t}}{\qpp{k'}\qpp{l'}} \qp{a^2q^{\#}}{t} OP[j,k+l+t-j]
    \end{align*}
    where we again write $k^\prime=j-k$ and $l^\prime=j-l$, and $M$, $M'$ are of
    the form $a^{\#}q^{\#\#}$. Note that the non-zero summands occur only for
    $\max(0,j-k-l)\leq t\leq \min(k',l')$.
\end{lem}

\subsection{Some quantum algebra}
This section introduces notation and lemmas that are useful in manipulating
expressions of HOMFLY-PT partition functions.

Here we will use capital letters such as $A$ to denote tuples of integers (typically
non-negative) $A=(A_1,\dots, A_a)$, whose elements are indicated by subscripts
(sometimes more than one). We use corresponding lower-case letters $a:=\#A$ to denote
their cardinalities and the notation $|A|=\sum_{x\in A} A$. 
 If $A$ and $B$ are such that $\# A=\# B$, then $A\leq B$ means $A_i\leq B_i$
for $1\leq i\leq \# A$ and $0\leq A$ means that $0\leq A_i$ for $1\leq i\leq \#
A$.

We use two basic lemmas: the first one is used for splitting $|A|$ into two
smaller pieces, and another one is used for resummations where we have two sets
of summation indices $A$ and $B$ such that $|A|=|B|$.

\begin{lem}\label{lem-first}
Let $A=(A_1,\dots, A_a)\in \N^a$, and $0\leq t\leq |A|$. Then we have

\[
   \qb{|A|}{A}= \sum_{\substack{At\in \N^a}
   }
    q^{\#}\qb{t}{At}\qb{|A|-t}{At^c}.
\]
where $At^c=(At^c_1,\dots, At^c_a):=(A_1-At_1,\dots, A_a-At_a)$.

\end{lem}
The non-zero contributions to the above sum are for tuples $At:=(At_1,\dots,
At_a)$ such that $|At|=t$ and $At^c\in \N^a$, since otherwise the multinomial
coefficients would be zero. The notation $At$ and $At^c$ is intended to indicate
that the entries of these tuple refine the entries of $A$ into components that
do or do not contribute to $t$, respectively.  Lemma~\ref{lem-first} is a
special case of the following.

\begin{lem}[Lemma 4.6 \cite{KRSSlong}]\label{lemvlong}
Let $A=(A_1,\ldots,A_a)\in \N^a$, $B=(B_1,\ldots,B_b)\in \N^b$ with
$|A|=|B|$. Then we have:
\begin{equation}\label{forexp}
\!\!\!\qb{|A|}{B}=
\sum_{\substack{AB\in \N^{ab}\\
|AB_{*,\beta}| = B_\beta
}
}
q^{\#}
\qb{A_1}{AB_{1,1},\ldots,AB_{1,b}}
\cdots \qb{A_a}{AB_{a,1},\ldots,AB_{a,b}},
\end{equation}
\end{lem}
The non-zero contributions to the above sum are for tuples $AB\in \N^{ab}$ with
entries $AB_{\alpha,\beta}$, with $1\leq \alpha\leq a$ and $1\leq \beta\leq b$,
such that $|AB_{\alpha,*}| = A_\alpha$ and $|AB_{*,\beta}| = B_\beta$. Also note
that $|AB|=|A|=|B|$. Here the notation $AB$ is intended to indicate a refinement
of the entries of the tuple $A$ according to the contribution to the entries of
$B$. 

As a corollary, we get the following.
\begin{lem}\label{lem-resum}
Let $A\in \N^a$ and $B\in \N^b$ with $|A|=|B|$. Then
\[
\!\!\!\qb{|A|}{A}\qb{|B|}{B}=
\sum_{\substack{AB\in \N^{ab}\\
|AB_{\alpha,*}| = A_\alpha\\
|AB_{*,\beta}| = B_\beta
}
}
q^{\#}\qb{|A|}{AB}
\]
\end{lem}

Next we need the $q$-Pochhammer symbol $(x^2;q^2)_k=
\prod_{i=0}^{k-1} (1-x^2q^{2i})$, which generalises $(q^2;q^2)_k=\qpp{k}$.

\begin{lem}[{\cite[Lemma 4.5]{KRSSlong}}]\label{lemlong}
    For any $d_1,\ldots,d_k\ge 0$, we have:
    \begin{align*}\frac{(x^2;q^2)_{d_1+\ldots+d_k}}{\qpp{d_1}\cdots \qpp{d_k}}
    =& \!\!\!  \sum\limits_{\substack{\alpha_1+\beta_1=d_1\\\cdots\\\alpha_k+\beta_k=d_k}}\!\!\!\!\!\!\!\!  
    \frac{(-x^2 q^{-1})^{{\alpha_1+\ldots+\alpha_k}} q^{\alpha_1^2+\ldots+\alpha_k^2+2\sum_{i=1}^{k-1} \alpha_{i+1} (d_1+\ldots+d_i)}}{\qpp{\alpha_1}\cdots\qpp{\alpha_k}\qpp{\beta_1}\cdots\qpp{\beta_k}}      
    \end{align*}
    \end{lem}

Having recalled these basic lemmas, we now assemble them into more specialised
tools for rewriting generating functions.

\begin{lem}
    \label{lem:rewrite-one}
Given any function of the form
\[
   P= \sum_{\substack{A\in \N^a\\B\in \N^b\\t\in \N\\
    t\leq |A|,\; t\leq |B|}}
    \qb{|A|}{A}\qb{|B|}{B} f(A,B,t)
\]
we can instead sum over $ABt\in \N^{ab}$, $At^c\in \N^a$, and $Bt^c\in \N^b$, with
$|ABt|=t$, $|At^c|=|A|-t$, and $|Bt^c|=|B|-t$ and write
\[
 P=   \sum_{\substack{ABt\in \N^{ab}\\At^c\in \N^a \\ Bt^c\in \N^b}}
 \qb{t}{ABt}\qb{|A|-t}{At^c}\qb{|B|-t}{Bt^c} f'(ABt,At^c,Bt^c)
\]
where $f'$ arises from substitutions and a monomial factor from $f$.
\end{lem}
\begin{proof}
We first apply Lemma~\ref{lem-first} and get
\[\qb{|A|}{A} =  \sum_{\substack{At\in \N^a}
}
 q^{\#}\qb{t}{At}\qb{|A|-t}{At^c}\] 
Since we also have
$|B|=b\geq t=|At|$, we can apply Lemma~\ref{lemvlong} to the tuples $(|B|-t,At_1,\dots,At_a)$ and $B$. Thus we get

\[\qb{|B|}{B} = \sum_{\substack{ABt\in \N^{ab}\\ Bt^c\in \N^b}} q^{\#}
\prod_{\alpha=1}^a\qb{At_\alpha}{ABt_{\alpha,*}}\cdot\qb{|B|-t}{Bt^c} \]
with constraints described in Lemma~\ref{lemvlong}. In combination, these
expressions achieve the desired rewriting of $P$ as far as multinomial
coefficients are concerned. To rewrite $f(A,B,t)$, note that
every entry of $A$ (resp. $B$) can be expressed as the sum of an entry of $At^c$
(resp $Bt^c$) and $b$ (resp. $a$) entries from the tuple $ABt$. Moreover,
$t=|ABt|$. Thus $f(A,B,t)$ can be re-expressed as a function of
the entries of $ABt$, $At^c$, and $Bt^c$. The function $f'(A',B',T')$ results from this
re-expression, also taking into account the monomial scaling $\#$. 
\end{proof}

\begin{lem}
    \label{lem:rewrite-two}
   Given any function of the form
    \[
        P= \sum_{\substack{A'\in \N^a,\; B'\in \N^b\\C\in \N^c,\; D\in \N^d\\
        |C|-|A'|=|D|-|B'|=s\geq 0}}
        \qb{|C|}{C}\qb{|D|}{D} f(A',B',C,D)
    \]
    we can instead sum over $A'C\in \N^{ac}$, $B'D\in \N^{bd}$, and $S\in \N^{cd}$ and write
    \[
        P= \sum_{\substack{A'C\in \N^{ac}\\B'D\in \N^{bd}\\S\in \N^{cd}}}
        \qb{s}{S} \prod_{\alpha=1}^a \qb{A'_\alpha}{A'C_{\alpha,*}}\cdot \prod_{\beta=1}^b \qb{B'_\beta}{B'D_{\beta,*}} f'(A'C,B'D,S)
       \]
       where $f'$ arises from substitutions and a monomial factor from $f$.
    \end{lem}
    \begin{proof}
    Since $|C|=s+|A'|$ with $s\geq 0$, we can apply Lemma~\ref{lemvlong} to the tuples $(s,A'_1,\dots, A'_a)$ and $C$:
    
    \[\qb{|C|}{C} =  \sum_{\substack{S'\in \N^c\\ A'C\in \N^{ac}}} q^{\#}\qb{s
}{S'} \prod_{\alpha=1}^a \qb{A'_\alpha}{A'C_{\alpha,*}}\] The
    non-zero contributions to this sum have $|S'|=s$, and so we also have
    $|D|=|S'|+|B'|$. Now we apply Lemma~\ref{lemvlong} to the tuples $(S'_1,\dots, S'_c,B'_1,\dots,B'_b)$ and $D$
    \[\qb{|D|}{D} =  \sum_{\substack{
    S'D\in \N^{cd}\\
    B'D\in \N^{bd}}} q^{\#} \prod_{\gamma=1}^c \qb{S'_\gamma}{S'D_{\gamma,*}}\cdot
    \prod_{\beta=1}^b \qb{B'_\beta}{B'D_{\beta,*}}\] By combining these two
    equations, setting $S:=S'D$, and re-expressing $f$ in these new summation
    indices, we obtain the claimed expression for $P$.
    \end{proof}

    The two previous lemmas combine in the following useful way.
    \begin{cor}
        \label{cor:rewrite}
        Given any function of the form
    \[P=
        \sum_{\substack{A\in \N^a,\; B\in \N^b\\C\in \N^c,\; D\in \N^d\\t\in \N\\
        |A|+|D|=|B|+|C|\\
        0\leq |A|-t\leq |C|\\
        0\leq |B|-t\leq |D|        }} \qb{|A|}{A}\qb{|B|}{B}\qb{|C|}{C}\qb{|D|}{D}
        f(A,B,C,D,t) \] 
        we can instead sum over $S\in \N^{cd}$, $ABt\in \N^{ab}$, $A'C\in \N^{ac}$, and $B'D\in \N^{bd}$, and write
    \[P=
    \sum_{\substack{S\in \N^{cd},\; ABt\in \N^{ab} \\A'C\in \N^{ac},\; B'D\in \N^{bd}}}
    \qb{|S|}{S}\qb{t}{ABt}\qb{|A'C|}{A'C}\qb{|B'D|}{B'D} f'(S,ABt,A'C,B'D)
\]
where  $|ABt|=t$, $|S|=|C|-|A|+t=|D|-|B|+t$, $|A'C|=|A|-t$, $|B'D|=|B|-t$, and $f'$ arises from substitutions and a monomial factor from $f$.
    \end{cor}
    \begin{proof}
        Apply Lemma~\ref{lem:rewrite-one} and then Lemma~\ref{lem:rewrite-two} with $A'=At^c$ and $B'=Bt^c$. 
    \end{proof}

    \subsection{RI-RI gluing}
Let $\tau_1$ and $\tau_2$ be tangles of type $RI$. Suppose for now that both
tangles have generating functions in quiver form. From now on we
use symbols $A_{1}$ and $A_2$ instead of $\textbf{d}_a$ for vectors of active summation
indices, and $I_{1}$ and $I_2$ instead of $\textbf{d}_i$ for vectors of inactive summation
indices. Then \eqref{eq:qf} for $\tau_1$ and $\tau_2$ take the form:

\begin{align}
    \label{eq:RIRI-input}
    P(\tau_1)& =
    \sum_{\substack{A_1\in \N^{a_1}\\I_1\in \N^{i_1} }} M_1 \qb{|A_1|}{A_1}\qb{|I_1|}{I_1} \qpp{|A_1|+|I_1|}^{-c_1} RI[|A_1|+|I_1|,|A_1|]
\\
\nonumber P(\tau_2) &=
\sum_{\substack{A_2\in \N^{a_2}\\I_2\in \N^{i_2} }} M_2 \qb{|A_2|}{A_2}\qb{|I_2|}{I_2} \qpp{|A_2|+|I_2|}^{-c_2} RI[|A_2|+|I_2|,|A_2|]   
     \end{align} 
        where $M_1$ and $M_2$ stand for a monomials of the
        form $(-q)^\# a^\# q^{\#\#}$, with exponents depending linearly, resp.
        quadratically on the summation indices in the tuples $A_i$ and $I_i$,
        and $c_i$ is the number of closed components in $\tau_i$. 
        
    Then we use Lemma~\ref{lem:RIRI} to compute the following expression for
$P(\tau_1+ \tau_2)$:

\small
\begin{gather*}\label{forRIRI} 
    \sum_{\substack{
    A_1\in \N^{a_1},\; I_1\in \N^{i_1}\\
    A_2\in \N^{a_2},\; I_2\in \N^{i_2}\\
    t\in \N\\
    |A_1|+|I_1|=|A_2|+|I_2|=:j
}}  \hspace{-1cm} M \qb{|A_1|}{A_1}\qb{|A_2|}{A_2}\qb{|I_1|}{I_1}\qb{|I_2|}{I_2}\qb{|I_1|+|I_2|}{|I_1|,|I_2|}
\qb{|I_1|+|I_2|-t}{t, |I_1|-t , |I_2|-t  } \qpp{j}^{-c} RI[j,|A_1|-|I_2|+t]
\end{gather*} 
\normalsize

where we write $c:=c_1+c_2$ and $M$ denotes a monomial in $a$ and $q$ with
exponents linear, resp. quadratic in the summation indices. Note that non-zero summands only occur for
$0\leq |I_1|-t\leq |A_2|$ and $0\leq |I_2|-t\leq |A_1|$. Thus we can apply
Corollary~\ref{cor:rewrite} with $A=I_1$, $B=I_2$, $C=A_2$, and $D=A_1$. After
contracting the quantum multinomial coefficients we obtain

\small
\begin{equation}
    \label{eq:RIRI}
    P(\tau_1+ \tau_2) = \hspace{-.5cm} \sum_{\substack{
    S\in \N^{a_1a_2}\\I_1I_2t\in \N^{i_1i_2} \\I_1t^cA_2\in \N^{i_1a_2}\\ I_2t^cA_1\in \N^{a_1i_2}
}}  \hspace{-.5cm}
M' \qb{|S|}{S}
\qb{|I_1 I_2 t|+|I_1 t^c A_2|+|I_2 t^c A_1|}{I_1 I_2 t, I_1 t^c A_2, I_2 t^c A_1} 
\qb{|I_1|+|I_2|}{|I_1|,|I_2|} \qpp{j}^{-c}
RI[j,|S|] 
\end{equation}
\normalsize

Here we have $|I_1 I_2t|=t$, $|I_1 t^c A_2|=|I_1|-t$, $|I_2 t^c A_1|=|I_2|-t$,
$|S|=|A_1|-|I_2|+t$, and $j=|I_1 I_2t|+|I_1 t^c A_2|+|I_2 t^c A_1|+|S|$. Note that 
\eqref{eq:RIRI} is almost in quiver form, except for the extra factor
$\qb{|I_1|+|I_2|}{|I_1|,|I_2|}$.

\begin{rem} The gluing rule can be made very explicit in triple
    notation. For example, if the input data \eqref{eq:RIRI-input} is given by
    \begin{equation*} P(\tau_1):= \left[
        \scalemath{0.75}{RI, \left( \begin{array}{c|c|c} -c_1 & S^1_+ & A^1_+ \\ \hline\hline -c_1 & S^1_- & A^1_- 
        \end{array} \right)},
         \scalemath{0.75}{
         \left( 
        \begin{array}{c|c} Q^1_{++} & Q^1_{+-} \\ \hline Q^1_{-+} & Q^1_{--} 
        \end{array} \right)} 
            \right], \qquad 
            P(\tau_2):= \left[
        \scalemath{0.75}{RI, \left( \begin{array}{c|c|c} -c_2 & S^2_+ & A^2_+ \\ \hline\hline -c_2 & S^2_- & A^2_- 
        \end{array} \right)},
         \scalemath{0.75}{
         \left( 
        \begin{array}{c|c} Q^2_{++} & Q^2_{+-} \\ \hline Q^2_{-+} & Q^2_{--} 
        \end{array} \right)} 
            \right]
    \end{equation*}
    then $P(\tau_1+ \tau_2)$ in the form of \eqref{eq:RIRI} (with an extra $q$-binomial coefficient) can be expressed via the data
    \begin{gather*}
        \left[ \scalemath{0.75}{RI, \left( \begin{array}{c|c|c} 
            -c & S^1_+\otimes S^2_+ & A^1_+\otimes A^2_+ \\ \hline\hline 
            -c & S^1_+\otimes S^2_- & A^1_+\otimes A^2_- \\ \hline
            -c & S^1_-\otimes S^2_+ & A^1_-\otimes A^2_+ \\ \hline
            -c & S^1_-\otimes S^2_- & A^1_-\otimes A^2_- 
        \end{array} \right)},
         \scalemath{0.75}{
         \left( 
        \begin{array}{c|c|c|c} 
            Q^1_{++}\otimes Q^2_{++} & Q^1_{++}\otimes Q^2_{+-} & Q^1_{+-}\otimes Q^2_{++} & Q^1_{+-}\otimes Q^2_{+-} \\ \hline 
            Q^1_{++}\otimes Q^2_{-+} & Q^1_{++}\otimes Q^2_{--} & Q^1_{+-}\otimes Q^2_{-+} & Q^1_{+-}\otimes Q^2_{--} \\ \hline 
            Q^1_{-+}\otimes Q^2_{++} & Q^1_{-+}\otimes Q^2_{+-} & Q^1_{--}\otimes Q^2_{++} & Q^1_{--}\otimes Q^2_{+-} \\ \hline 
            Q^1_{-+}\otimes Q^2_{-+} & Q^1_{-+}\otimes Q^2_{--} & Q^1_{--}\otimes Q^2_{-+} & Q^1_{--}\otimes Q^2_{--}
        \end{array}\right)+C} \right]
    \end{gather*}
    where the matrix $C$ contains certain off-diagonal correction terms. For
    example, if $P(\tau_1)$ and $P(\tau_2)$ have (inactive) quiver forms with $2$ active and $2$ inactive variables each, then
    \[C= \scalemath{0.5}{
        \left( \begin{array}{c c c c|c c c c |c c c c|c c c c} 
        0 & 0 & 0 & 1 & 0 & 1 & 0 & 1 & 0 & 1 & 0 & 1 & 0 & 0 & 0 & 0 \\
        0 & 0 & 0 & 0 & 0 & 1 & 0 & 1 & 0 & 0 & 0 & 0 & 0 & 0 & 0 & 0 \\ 
        0 & 0 & 0 & 0 & 0 & 0 & 0 & 0 & 0 & 1 & 0 & 1 & 0 & 0 & 0 & 0 \\
        1 & 0 & 0 & 0 & 0 & 0 & 0 & 0 & 0 & 0 & 0 & 0 & 0 & 0 & 0 & 0 \\ \hline 
        0 & 0 & 0 & 0 & 0 & 0 & 0 & 1 &-1 &-1 &-1 &-1 &-1 &-1 & 0 & 0 \\
        1 & 1 & 0 & 0 & 0 & 0 & 0 & 0 &-1 &-1 &-1 &-1 &-1 &-1 & 0 & 0 \\ 
        0 & 0 & 0 & 0 & 0 & 0 & 0 & 0 &-1 &-1 &-1 &-1 &-1 &-1 &-1 &-1 \\
        1 & 1 & 0 & 0 & 1 & 0 & 0 & 0 &-1 &-1 &-1 &-1 &-1 &-1 &-1 &-1 \\ \hline 
        0 & 0 & 0 & 0 &-1 &-1 &-1 &-1 & 0 & 0 & 0 & 1 &-1 & 0 &-1 & 0 \\
        1 & 0 & 1 & 0 &-1 &-1 &-1 &-1 & 0 & 0 & 0 & 0 &-1 & 0 &-1 & 0 \\ 
        0 & 0 & 0 & 0 &-1 &-1 &-1 &-1 & 0 & 0 & 0 & 0 &-1 &-1 &-1 &-1 \\
        1 & 0 & 1 & 0 &-1 &-1 &-1 &-1 & 1 & 0 & 0 & 0 &-1 &-1 &-1 &-1 \\ \hline 
        0 & 0 & 0 & 0 &-1 &-1 &-1 &-1 &-1 &-1 &-1 &-1 &-1 &-1 &-1 & 0 \\
        0 & 0 & 0 & 0 &-1 &-1 &-1 &-1 & 0 & 0 &-1 &-1 &-1 &-1 &-1 &-1 \\ 
        0 & 0 & 0 & 0 & 0 & 0 &-1 &-1 &-1 &-1 &-1 &-1 &-1 &-1 &-1 &-1 \\
        0 & 0 & 0 & 0 & 0 & 0 &-1 &-1 & 0 & 0 &-1 &-1 & 0 &-1 &-1 &-1 
    \end{array} \right)}
    \]
\end{rem}

\subsection{OP-UP gluing}
Let $\tau_1$ and $\tau_2$ be tangles of type $OP$ and $UP$ respectively. Suppose
for now that both tangles have generating functions in quiver form:

\begin{align}
    \label{eq:OPUP-input}
    P(\tau_1)& =
    \sum_{\substack{A_1\in \N^{a_1}\\I_1\in \N^{i_1} }} M_1 \qb{|A_1|}{A_1}\qb{|I_1|}{I_1} \qpp{|A_1|+|I_1|}^{-c_1} OP[|A_1|+|I_1|,|A_1|]
\\
\nonumber
 P(\tau_2) &=
\sum_{\substack{A_2\in \N^{a_2}\\I_2\in \N^{i_2} }} M_2 \qb{|A_2|}{A_2}\qb{|I_2|}{I_2} \qpp{|A_2|+|I_2|}^{-c_2} UP[|A_2|+|I_2|,|A_2|]   
     \end{align}

     Then we use Lemma~\ref{lem:OPUP} to compute the following expression for
     $P(\tau_1+ \tau_2)$:

     \small
\begin{gather*} 
    \sum_{\substack{
    A_1\in \N^{a_1},\; I_1\in \N^{i_1}\\
    A_2\in \N^{a_2},\; I_2\in \N^{i_2}\\
    t\in \N
    \\
    |A_1|+|I_1|=|A_2|+|I_2|=:j
}}  \hspace{-1.2cm}M \qb{|A_1|}{A_1}\qb{|A_2|}{A_2}\qb{|I_1|}{I_1}
\qb{|I_2|}{I_2}\qb{|I_2|+t}{t, |I_1|-t , |I_2|-|I_1|+t}
\frac{\qp{a^2q^{\#}}{|I_1|}}{\qpp{|I_1|}\qpp{j}^{c}}  UP[j,|A_2|-t]
\end{gather*} 
\normalsize

Note that the present quantum multinomial coefficients imply that the non-zero
contributions to the sum satisfy $0\leq |I_1|-t\leq |I_2|$ and $0\leq
|A_2|-t\leq |A_1|$, and so we can apply Corollary~\ref{cor:rewrite} with
$A=A_2$, $B=I_1$, $C=A_1$, and $D=I_2$. After contracting the quantum
multinomial coefficients we obtain the following expression for $P(\tau_1+ \tau_2)$.

\small
\begin{equation}
    \label{eq:OPUP} \sum_{\substack{
    S\in \N^{a_1i_2}\\A_2I_1t\in \N^{i_1a_2} \\A_2t^cA_1\in \N^{a_1a_2}\\ I_1t^cI_2\in \N^{i_1i_2}
}}  \hspace{-.5cm}
M' \qb{|A_2 t^c A_1|}{A_2 t^c A_1}
\qb{|A_2 I_1 t|+|I_1 t^c I_2|+|S|}{A_2 I_1 t, I_1 t^c I_2 , S  } 
\frac{\qp{a^2q^{\#}}{|I_1|}}{\qpp{|I_1|}\qpp{j}^{c}}
UP[j,|A_2 t^c A_1|] \end{equation} 
\normalsize

Here we have $|A_2 I_1t|=t$, $|A_2 t^c A_1|=|A_2|-t$, $|I_1 t^c I_2|=|I_1|-t$,
$|S|=|I_2|-|I_1|+t$, and $j=|A_2 I_1t|+|A_2 t^c A_1|+|I_1 t^c I_2|+|S|$. Note
that \eqref{eq:OPUP} is almost in quiver form, except for the extra factor
$\frac{\qp{a^2q^{\#}}{|I_1|}}{\qpp{|I_1|}}$ and possible an extra
factor of $\qpp{j}$, depending on the number of components of $\tau_1+ \tau_2$.

\subsection{OP-OP gluing}
Let $\tau_1$ and $\tau_2$ be tangles of type $OP$. Suppose
for now that both tangles have generating functions in quiver form:

\begin{align}
    \label{eq:OPOP-input}
    P(\tau_1)& =
    \sum_{\substack{A_1\in \N^{a_1}\\I_1\in \N^{i_1} }} M_1 
    \qb{|A_1|}{A_1}\qb{|I_1|}{I_1} \qpp{|A_1|+|I_1|}^{-c_1} OP[|A_1|+|I_1|,|A_1|]
\\
\nonumber
 P(\tau_2) &=
\sum_{\substack{A_2\in \N^{a_2}\\I_2\in \N^{i_2} }} M_2 
\qb{|A_2|}{A_2}\qb{|I_2|}{I_2} \qpp{|A_2|+|I_2|}^{-c_2} OP[|A_2|+|I_2|,|A_2|]   
     \end{align}

     Then we use Lemma~\ref{lem:OPOP} to compute the following expression for
     $P(\tau_1+ \tau_2)$:


\small
\begin{gather*} 
    \sum_{  \substack{  A_1\in \N^{a_1},\; I_1\in \N^{i_1}\\
    A_2\in \N^{a_2},\; I_2\in \N^{i_2}\\
    t\in \N\\
    |A_1|+|I_1|=|A_2|+|I_2|=:j
}}  \hspace{-1.3cm}M
\qb{|A_1|}{A_1}
\qb{|A_2|}{A_2}
\qb{|I_1|}{I_1}
\qb{|I_2|}{I_2}
\qb{|I_1|+|I_2|-t}{t,|I_1|-t,|I_2|-t}
\frac{\qpp{|I_1|+|I_2|-t} \qp{a^2q^{\#}}{t}}{\qpp{|I_1|}\qpp{|I_2|} \qpp{j}^{c}} 
 OP[j,|A_1|+|A_2|+t-j]
\end{gather*}
\normalsize
Note that the present quantum multinomial coefficients imply that the non-zero
contributions to the sum satisfy $0\leq |I_1|-t\leq |A_2|$ and $0\leq
|I_2|-t\leq |A_1|$, and so we can apply Corollary~\ref{cor:rewrite} with
$A=I_2$, $B=I_1$, $C=A_1$, and $D=A_2$. After contracting the quantum
multinomial coefficients we obtain the following expression for $P(\tau_1+ \tau_2)$.

\small
\begin{equation}
    \label{eq:OPOP} \sum_{\substack{
    S\in \N^{a_1a_2}\\I_1I_2t\in \N^{i_1i_2} \\I_2t^cA_1\in \N^{a_1i_2}\\ I_1t^cA_2\in \N^{i_1a_2}
}}  \hspace{-.5cm}
M' \qb{|S|}{S}
\qb{|I_1 I_2 t|+|I_1 t^c A_2|+|I_2 t^c A_1|}{ I_1 I_2 t, I_1 t^c A_2 , I_2 t^c A_1} 
\frac{\qpp{|I_1|+|I_2|-t} \qp{a^2q^{\#}}{t}}{\qpp{|I_1|}\qpp{|I_2|} \qpp{j}^{c}} 
 OP[j,|A_1|+|A_2|+t-j] \end{equation} 
\normalsize

Here we have $|I_1 I_2t|=t$, $|I_2 t^c A_1|=|I_2|-t$, $|I_1 t^c A_2|=|I_1|-t$,
$|S|=|A_1|+|A_2|+t-j$, and $j=|I_1 I_2t|+|I_2 t^c A_1|+|I_1 t^c A_2|+|S|$. Note
that \eqref{eq:OPUP} is almost in quiver form, except for the extra factor
$\frac{\qpp{|I_1|+|I_2|-t}\qp{a^2q^{\#}}{t}}{\qpp{|I_1|}\qpp{|I_2|}}$ and possible an extra
factor of $\qpp{j}$, depending on the number of components of $\tau_1+ \tau_2$.

\subsection{Proof of the main theorem}
In the last section we have computed that if two $4$-ended tangles $\tau_1$ and
$\tau_2$ as in Theorem~\ref{thm:mainthm} have HOMFLY-PT partition functions in
quiver form, then the partition function for $\tau_1+ \tau_2$ is very
close to being in quiver form. We will now refine this observation to assemble a
proof of Theorem~\ref{thm:mainthm}. We consider the cases $RI$-$RI$ and
$OP$-$UP$ separately. 

\begin{proof}[Proof of Theorem~\ref{thm:mainthm} for RI-RI gluing.] Suppose that
    $\tau_1,\tau_2\in \QTf$ are both of type $RI$. Let $c_i$ denote the number
    of closed components of $\tau_i$ and observe that $\tau_1+\tau_2$ has
    $c_1+c_2$ closed components. Now there are four connectivity configurations
    to consider.
    \begin{enumerate}
        \item Both $\tau_1$ and $\tau_2$ are of type $RI_{par}$. Then
        $\tau_1+\tau_2$ will also be of type $RI_{par}$. By assumption,
        both $P(\tau_1)$ and $P(\tau_2)$ may be assumed to be in inactive quiver
        form, i.e. the expressions in \eqref{eq:RIRI-input} may be assumed have an
        additional factor $\qpp{|I_i|}$ each. The partition function for the
        glued tangle, rewritten as in \eqref{eq:RIRI}, is manifestly in quiver
        form, except that it carries an extra factor of 
        \[\qpp{|I_1|}\qpp{|I_2|}\qb{|I_1|+|I_2|}{|I_1|,|I_2|}=\#
        \qpp{|I_1|+|I_2|}.\] Since the length of this $q$-Pochhammer symbol
        $|I_1|+|I_2|$ is greater than the sum $|I_1|+|I_2|-t$ of the new
        inactive variables, it can be shortened (at the expense of splitting the
        summation indices from the set $I_1I_2t$ via Lemma~\ref{lemlong}) to
        bring $P(\tau_1 + \tau_2)$ into inactive quiver form, as required
        for a tangle of type $RI_{par}$. Thus $\tau_1+\tau_2 \in\QTf$. 
        \item If $\tau_1$ is of type $RI_{par}$ but $\tau_2$ is of type
        $RI_{cr}$, then $\tau_1 + \tau_2$ will also be of type $RI_{cr}$.
        Lemma~\ref{lem:cr-equic-cond} shows that $\tau_2\in \QTf \iff
        R\tau_2\in \QTf$, but $R\tau_2$ is of type $RI_{par}$, so we can use
        the first case to deduce $R(\tau_1+\tau_2)=\tau_1+
        R\tau_2\in \QTf$. Now, again by Lemma~\ref{lem:cr-equic-cond}, this is
        equivalent to $\tau_1+ \tau_2\in \QTf$.
        \item Now suppose $\tau_1$ is of type $RI_{cr}$ and $\tau_2$ is of type
        $RI_{par}$. We write
        \[R(\tau_1+\tau_2)=\tau_1+ R\tau_2 = R\tau_1+
        r_{h}(\tau_2)\] where $r_{h}(\tau_2)$ is the result of rotating
        $\tau_2$ by $\pi$ in the $EW$-axis. By Lemma~\ref{lem:rotatetwo}
        $r_{h}(\tau_2)$ is again in $\QTf$ and of type $RI_{par}$, and so is
        $R\tau_1$. Now we apply the first case to see that
        $R(\tau_1+\tau_2)\in \QTf$ and conclude with a final application
        of Lemma~\ref{lem:cr-equic-cond} that $\tau_1+\tau_2\in \QTf$
        \item Both $\tau_1$ and $\tau_2$ are of type $RI_{cr}$, then perform a
        Reidemeister II move in the gluing region to write 
         \[\tau_1+\tau_2 = R\tau_1+ R^{-1}r_{h}(\tau_2).\] Now
        both $R\tau_1$ and $R^{-1}r_{h}(\tau_2)$ are of type $RI_{par}$ and
        in $\QTf$ by assumption, and we again use the first case to conclude
        $\tau_1+\tau_2\in \QTf$.
    \end{enumerate}
This completes the proof in the case of $RI$-$RI$ gluing.
\end{proof}

\begin{proof}[Proof of Theorem~\ref{thm:mainthm} for OP-UP gluing.] Suppose that
    $\tau_1,\tau_2\in \QTf$ with $\tau_1$ of type $OP$ and $\tau_2$ of type
    $UP$. Furthermore, let $c_i$ denote the number of closed components of
    $\tau_i$ and set $c=c_1+c_2$. There are again four connectivity
    configurations to consider.
    \begin{enumerate}
        \item If $\tau_1$ is of type $OP_{ud}$ and $\tau_2$ of $UP_{par}$, then
        $\tau_1+ \tau_2$ is of type $UP_{par}$ with $c+1$ components. By
        assumption, both $P(\tau_1)$ and $P(\tau_2)$ may be assumed to be in
        active quiver form, i.e. expressions in \eqref{eq:OPUP-input} may be
        assumed have an additional factor $\qpp{|A_i|}$ each. The partition
        function $P(\tau_1+ \tau_2)$ for the glued tangle, rewritten as in
        \eqref{eq:OPUP}, is manifestly in quiver form, except that it carries an
        extra factor of 
\[\frac{\qp{a^2q^{\#}}{|I_1|}\qpp{|A_1|}\qpp{|A_2|}\qpp{j}}{\qpp{|I_1|}} =
      \qp{a^2q^{\#}}{|I_1|}\qpp{|A_1|}\qp{q^{2+\#}}{j-|I_1|}\qp{q^{2+\#}}{t}\qpp{|A_2|-t}
      \] The last of the $q$-Pochhammer symbols shown on the right-hand side is
      what is necessary to put $P(\tau_1+ \tau_2)$ in active quiver form
      after using all other symbols to split summation indices. Thus we have
      $\tau_1+ \tau_2\in \QTf$.
        \item If $\tau_1$ is of type $OP_{lr}$ and $\tau_2$ of $UP_{par}$, then
        $\tau_1+ \tau_2$ is of type $UP_{par}$ with $c$ components. We
        proceed as in the first case, with the notable difference that
        $P(\tau_1)$ may be assumed to be in inactive quiver form. This provides
        a $q$-Pochhammer factor $\qpp{|I_1|}$ that we use in place of the factor
        $\qpp{j}$ to cancel the denominator in the extra factor in
        \eqref{eq:OPUP}. Thus we can write $P(\tau_1+ \tau_2)$ in active
        quiver form and deduce $\tau_1+ \tau_2\in \QTf$.
        
        \item If $\tau_1$ is of type $OP_{ud}$ and $\tau_2$ of $UP_{cr}$, then
        $\tau_1+ \tau_2$ is of type $UP_{par}$ with $c$ components. We can
        now perform a Reidemeister II move in the gluing region, to write
        \[\tau_1+ \tau_2 = R\tau_1+ L^{-1}\tau_2\] Here
        $R\tau_1$ is now of type $UP_{par}$ and $L^{-1}\tau_2$ is $\tau_2$
        with a crossing attached on the left, i.e. a rotated version of a tangle
        of type $OP_{lr}$. Now the $\pi$-rotation $r_{v}$ reduces the situation
        to the already established second case.
        
        \item If $\tau_1$ is of type $OP_{lr}$ and $\tau_2$ of $UP_{cr}$, then
        $\tau_1+ \tau_2$ is of type $UP_{cr}$ with $c$ components. By
        Lemma~\ref{lem:cr-equic-cond} it suffices to prove that the type
        $OP_{lr}$ tangle $R(\tau_1+ \tau_2) = \tau_1+ R\tau_2$ is in
        $\QTf$. For this we note that $\tau_1$ and $R\tau_2$ are both of type
        $OP_{lr}$ and in $\QTf$, so $P(\tau_1)$ and $P(R\tau_2)$ may be
        assumed to be in inactive quiver form, i.e. the expressions in
        \eqref{eq:OPOP-input} may be assumed have an additional factor
        $\qpp{|I_i|}$ each. The partition function for the glued tangle,
        rewritten as in \eqref{eq:OPOP}, is manifestly in quiver form, except
        that it carries an extra factor of 
        \[\qpp{|I_1|+|I_2|-t}\qp{a^2q^{\#}}{t}\] After using $\qp{a^2q^{\#}}{t}$
        to split some summation indices, the expression for $P(\tau_1+
        R\tau_2)$ is in inactive quiver form, as required for an $OP_{lr}$ tangle in $\QTf$.

    \end{enumerate}
This completes the proof in the case of $OP$-$UP$ gluing.
\end{proof}

\subsection{Size estimates for glued quiver data}
\label{sec:size}
If $\tau_1,\tau_2\in \QT$ admit generating functions in quiver form and suppose
the sum $\tau_1+\tau_2$ is defined, then it is plausible to
expect that the number of summation indices in a quiver form for
$P(\tau_1+ \tau_2)$ is at worst bilinear in the inputs sizes. The purpose of
this section is to refine and prove such a statement.

\begin{rem}
    \label{rem:deactivate}
Any generating function in active/inactive quiver form can be expanded into
quiver form by using the extra $q$-Pochhammer symbol to double the number of
active/inactive summation indices via Lemma~\ref{lemlong}. As explained in
\cite[Lemma 4.11]{SW} this corresponds to an identity of the following type in
triple notation
\[
    \left[\scalemath{0.75}{ UP, \left( \begin{array}{c|c|c} -c & S_+ +1 & A_+ \\ \hline -c & S_+ & A_+ \\ \hline\hline -c & S_- & A_- 
    \end{array} \right)},
     \scalemath{0.75}{
     \left( 
    \begin{array}{c|c|c} 
        Q_{++} +1 & Q_{++}+L & Q_{+-} \\ \hline 
        Q_{++} +U & Q_{++} & Q_{+-} \\ \hline 
        Q_{-+} & Q_{-+} & Q_{--} 
    \end{array} \right)} 
    \right]
    \;=\;
    \left[
\scalemath{0.75}{UP, \left( \begin{array}{c|c|c} 1-c & S_+ & A_+ \\ \hline\hline -c & S_- & A_- 
\end{array} \right)},
 \scalemath{0.75}{
 \left( 
\begin{array}{c|c} Q_{++} & Q_{+-} \\ \hline Q_{-+} & Q_{--} 
\end{array} \right)} 
    \right]
\]
where $1$ indicates matrices with all entries equal to $1$, and $U$ resp. $L$
denote matrices that have entries $1$ strictly above resp. below the diagonal and zeros elsewhere.
\end{rem}

\begin{prop}\label{prop:bounds} Suppose that $\tau_1,\tau_2\in \QT$ are gluable of types
$RI_{par}$-$RI_{par}$, $OP_{lr}$-$UP_{par}$, or $OP_{lr}$-$OP_{lr}$. Use
Lemma~\ref{lemlong} to convert their generating functions in active/inactive
quiver form into generating functions of quiver form \eqref{eq:qf} with $a_1$
resp. $a_2$ active and $i_1$ resp. $i_2$ inactive summation indices. Then
$P(\tau_1+ \tau_2)$ can be brought into quiver form with a total of
$(a_1+i_1)(a_2+i_2)$ summation indices.
\end{prop}
\begin{proof}
    We will only consider the cases $RI_{par}$-$RI_{par}$ and
$OP_{lr}$-$UP_{cr}$, since $OP_{lr}$-$OP_{lr}$ is entirely analogous to the
former. 

In the case when we add $\tau_1,\tau_2\in \QTf$ of type $RI_{par}$, we may
assume that their generating functions in inactive quiver form
\eqref{eq:RIRI-input} have $a_1+i_1/2$ and $a_2+i_2/2$ summation indices
respectively (Rewriting into quiver form via Lemma~\ref{lemlong} doubles the
numbers of inactive indices.). When computing $P(\tau_1+ \tau_2)$ along
the steps of the proof of Theorem~\ref{thm:mainthm} the expression
\eqref{eq:RIRI} will have $(a_1+i_1/2)(a_2+i_2/2)$ summation indices. After
splitting the summation indices from the set $I_1I_2t$, we arrive at an inactive
quiver form with $a_1a_2$ active summation indices and $a_1i_2/2+i_1a_2/2 +
i_1i_2/2$ inactive summation indices. By using Lemma~\ref{lemlong} again, this
is converted into an expression in quiver form with $a_1a_2$ active summation
indices and $a_1i_2+i_1a_2 + i_1i_2$ inactive summation indices. The total
number is $(a_1+i_1)(a_2+i_2)$ as claimed.

In the case of $OP_{lr}$-$UP_{par}$ we start with generating functions in
inactive/active quiver form with $a_1+i_1/2$ and $a_2/2+i_2$ summation indices
respectively. At the stage of \eqref{eq:OPUP} we see a total number of
$(a_1+i_1/2)(a_2/2+i_2)$ summation indices, and then use extra $q$-Pochhammer
symbols of length $t$ and $|I_1|$ to obtain an active quiver form with
$a_1a_2/2$ active and $i_1a_2+a_1i_2 +i_1i_2$ summation indices. A final
application of Lemma~\ref{lemlong} puts $P(\tau_1+ \tau_2)$ into quiver
form with $a_1a_2$ active and $i_1a_2+a_1i_2 +i_1i_2$ inactive indices. The total
number is again $(a_1+i_1)(a_2+i_2)$. 
\end{proof}

The minimal numbers of summation indices of generating functions for added
tangles are often smaller that suggested by Proposition~\ref{prop:bounds}. For
example, the generating function for a single crossing has $1$ active and $1$
inactive summation index. However, recall the following result.

\begin{prop}[{\cite[Proposition 3.5]{SW}}]
    \label{prop:TR} If $\tau\in \Tf$ and $P(\tau)$ is in
quiver form with $a+i$ summation indices, then $P(T\tau)$ and $P(R\tau)$ can
also be brought into quiver form with $(a+i)+i$ and $a+(a+i)$ summation indices
respectively.
\end{prop}

In particular, the sizes of the resulting generating functions are smaller than the
product of the input sizes. Proposition~\ref{prop:bounds} can also be combined
with Proposition~\ref{prop:TR} used to deduce size estimates in other gluing
situations. We just give one example.

\begin{cor} Suppose that $\tau_1,\tau_2\in \QT$ are gluable of types
    $RI_{par}$-$RI_{cr}$ and suppose that $P(\tau_1)$ and $P(R^{-1}\tau_2)$
    have been brought into quiver form with $a_1+i_1$ and $a_2+(i_2-a_2)$
    summation indices respectively by applying Lemma~\ref{lemlong} to their
    active quiver form expressions. Then $P(\tau_2)$ and $P(\tau_1+
    \tau_2)$ can be brought into quiver form with $a_2+i_2$ summation indices
    and $a_1a_2 + (a_1i_2+i_1i_2)$ summation indices respectively. 
\end{cor}
\begin{proof} Proposition~\ref{prop:TR} implies that $P(\tau_2)$ can be
expressed in quiver form of size $a_2+i_2$. For $P(\tau_1+ \tau_2)$ we
first apply Proposition~\ref{prop:bounds} to $\tau_1$ and $R^{-1}\tau_2$ and
obtain a quiver form for $P(\tau_1+ R^{-1}\tau_2)$ of size $a_1a_2 +
(a_1i_2 + i_1i_2 - a_1a_2 )$. After applying another operation $R$, we conclude
with $P(\tau_1+ \tau_2)$ in quiver form with $a_1a_2 + (a_1i_2+i_1i_2)$
summation indices.
\end{proof}

To get a general size estimate we need the following definition. 

\begin{defn} For $\tau\in \QTf$, the quiver complexity $\qc(\tau)\in \N$
is defined as the minimum of the following numbers
    \begin{itemize}
        \item  $2a+i$ such that $P(\tau)$ has an active quiver form with $a$ active and $i$
        inactive variables, if $\tau$ is of type $UP_{par}$ or $OP_{par}$,
        \item $a+2i$ such that $P(\tau)$ has an inactive quiver form with $a$ active and $i$
        inactive variables, if $\tau$ is of type $OP_{lr}$ or $RI_{par}$.
    \end{itemize}
    In other words, the complexity is the minimal size of a quiver form for
    $P(\tau)$, that can be obtained from an active/inactive quiver form. We call
    such quiver forms tight. For the remaining cases $UP_{cr}$ or $RI_{cr}$ we
    define $\qc(\tau)$ as the minimum of
    \begin{itemize}
    \item $2\qc(\tau')$, where $\tau'\in\{T\tau,T^{-1}\tau,R\tau,R^{-1}\tau\}$.
    \end{itemize}
\end{defn}
The next observation follows directly from Proposition~\ref{prop:TR}.

\begin{lem} For any $\tau \in \QTf$ the generating function $P(\tau)$ admits a
quiver form with at most $\qc(\tau)$ summation indices.
\end{lem}

\begin{lem} For $\tau\in \QTf$ have $\qc(C\tau)\leq 4 \qc(\tau)$ for any $C\in \{T, T^{-1},R,R^{-1}\}$.
\end{lem}
\begin{proof} This follows from the definition if $C\tau$ is of type $UP_{cr}$
    or $RI_{cr}$. All other cases follow from \cite[Lemmas 4.5, 4.6, 4.7]{SW}.
    We illustrate one of the most interesting cases. Suppose that $C=R^{-1}$ and
    $\tau$ is of type $UP_{cr}$ with complexity $4a+2i$ witnessed by an active
    quiver form of size $a+i$ for $\tau':=T^{-1}\tau$.
    %
    %
Then \cite[Lemmas 4.6, 4.7]{SW} imply that $R^{-1}\tau=R^{-1}T^{-1}T^2\tau'$ has an inactive quiver form with
    $(2a+3i)+(a+2i)$ summation indices, and we get
    \[\qc(R^{-1}\tau)=\qc(R^{-1}T^{-1}T^2\tau')\leq 4a+7i \leq 16a+8i=4\qc(\tau).\vspace{-.6cm}\]
\end{proof}
It is likely that the constant can be improved from $4$ to $2$ using an improved
version of \cite[Lemmas 4.6]{SW} for operations of the form $R^{-1}T$ etc.

\begin{prop}\label{prop:generalbound} Suppose that $\tau_1,\tau_2\in \QTf$ are gluable. Then
\[\qc(\tau_1+ \tau_2)\leq 16 \qc(\tau_1)\qc(\tau_2).\]
\end{prop}
\begin{proof} For the cases from Proposition~\ref{prop:bounds} we already know
$\qc(\tau_1+ \tau_2)=\qc(\tau_1)\qc(\tau_2)$. In the proof of
Theorem~\ref{thm:mainthm}, the remaining five cases are deduced from the those
by performing rewrites with at most two crossings, or by absorbing extra
$q$-Pochhammer symbols at the expense of at most quadrupling the number of
summation indices. In all cases, the complexity bound is observed. 
\end{proof}



\section{Case study on two pretzel tangles}
\label{sec:examples}
In this section we give examples of quiver form expressions of the generating
functions of the $(2,3)$ and $(2,-3)$ pretzel tangles of orientation type
$OP$-$UP$. Both result from adding two rational tangles whose quiver form
generating functions require $3+1$ and $2+1$ summation indices respectively.

An optimistic reading of Section~\ref{sec:size}
suggest\footnote{Proposition~\ref{prop:bounds} does not cover the
$OP_{ud}$-$UP_{cr}$-case, but the stated upper bound can be deduced for the
$(2,-3)$ pretzel tangle. A similar argument for the $(2,3)$ pretzel tangle
produces a worse upper bound, but still a better one than Proposition~\ref{prop:generalbound}.} an upper bound of $6+6$ summation indices for
quiver forms and $3+6$ summation indices for active quiver forms. In fact, both
the $(2,3)$ and the $(2,-3)$ pretzel tangle admit smaller expressions,
interestingly of different sizes: $5+6$ and $3+6$ in quiver form, $5+3$ and
$3+3$ in active quiver form.

\subsection{The (2,3)-pretzel tangle}
The following shows the quiver form generating function data for the skein
module element of the $(2,3)$-pretzel tangle. The diagram on the right shows the
$11$ monomials, $6$ active and $5$ inactive, in the $1$-colored skein
evaluation, which gives a lower bound on the size of any quiver describing the
generating function. This lower bound is achieved by the following data, which was found experimentally.
\[
    \begin{tikzpicture} [scale=.5,anchorbase,xscale=-1,yscale=-1]
        \draw[thick] (1,1) \pu (0,2);
        \draw[thick] (1,2) \pu (0,3);
        \draw[thick] (1,3) \pu (0,4);
        \draw[thick] (3,2.5) \pu (2,3.5);
        \draw[thick] (3,1.5) \pu (2,2.5);
        \draw[white, line width=.15cm] (0,1) \pu (1,2);
        \draw[white, line width=.15cm] (0,2) \pu (1,3);
        \draw[white, line width=.15cm] (0,3) \pu (1,4);
        \draw[white, line width=.15cm] (2,2.5) \pu (3,3.5);
        \draw[white, line width=.15cm] (2,1.5) \pu (3,2.5);
        \draw[thick] (0,1) \pu (1,2);
        \draw[thick] (0,2) \pu (1,3);
        \draw[thick] (0,3) \pu (1,4);
        \draw[thick] (2,2.5) \pu (3,3.5);
        \draw[thick] (2,1.5) \pu (3,2.5);
        \draw[thick,<-] 
        (0,0) to (0,1);
        \draw[thick]  (1,4) \ur (1.5,4.5) \rd (2,4) \pd (2,3.5) 
        (2,1.5) \pd (2,1) \dl (1.5,.5) \lu (1,1);
        \draw[thick]  (0,4) to (0,5);
        \draw[thick,<-] (3,0) to (3,1.5); 
        \draw[thick] (3,3.5) to (3,5);
        \end{tikzpicture}    
\qquad
\scalemath{0.65}{\left(\begin{array}{c|c|c}
    0& 2 & 1 \\ \hline 
    0&1 & 1 \\ \hline 
    0&0 & 1 \\ \hline 
    0&1 & 1 \\ \hline 
    0& 0 & 1 \\ \hline 
    0& -1 & 1 \\ \hline \hline
    0&0 & 1 \\ \hline 
    0&-1 & 1 \\ \hline 
    0&-2 & 1 \\ \hline 
    0&-1 & -1 \\ \hline 
    0&-3 & -1 
    \end{array}\right)}
   ,
   \scalemath{0.65}{\left(\begin{array}{c|c|c|c|c|c|c|c|c|c|c} 
    1 & 0 & -1 & 0 & -1 & -2 & 0 & -1 & -2 & 2 & 0 \\ \hline 
    0 & 0 & -1 & 0 & -1 & -2 & -1 & -1 & -2 & 1 & 0 \\ \hline 
    -1 & -1 & -1 & -1 & -1 & -2 & -1 & -1 & -2 & 1 & 0 \\ \hline 
    0 & 0 & -1 & 0 & -1 & -2 & 0 & -1 & -2 & 2 & 0 \\ \hline 
    -1 & -1 & -1 & -1 & -1 & -2 & -1 & -1 & -2 & 1 & 0 \\ \hline 
    -2 & -2 & -2 & -2 & -2 & -2 & -1 & -1 & -2 & 1 & 0 \\ \hline 
    0 & -1 & -1 & 0 & -1 & -1 & 0 & -2 & -2 & 1 & -1 \\ \hline 
    -1 & -1 & -1 & -1 & -1 & -1 & -2 & -1 & -2 & -1 & -1 \\ \hline 
    -2 & -2 & -2 & -2 & -2 & -2 & -2 & -2 & -2 & 0 & -1 \\ \hline 
    2 & 1 & 1 & 2 & 1 & 1 & 1 & -1 & 0 & 3 & 1 \\ \hline 
    0 & 0 & 0 & 0 & 0 & 0 & -1 & -1 & -1 & 1 & 1 
    \end{array} \right)}
     \quad
\begin{tikzpicture} [scale=.4,anchorbase]
\draw[thin,opacity=.5] (-4.5,-1) to (3,-1);
\draw[thin,opacity=.5] (-4.5,0) to (3,0);
\draw[thin,opacity=.5] (-4.5,1) to (3,1);
\draw[thin,opacity=.5] (-4,-1.5) to (-4,1);
\draw[thin,opacity=.5] (-3,-1.5) to (-3,1);
\draw[thin,opacity=.5] (-2,-1.5) to (-2,1);
\draw[thin,opacity=.5] (-1,-1.5) to (-1,1);
\draw[thin,opacity=.5] (0,-1.5) to (0,1);
\draw[thin,opacity=.5] (1,-1.5) to (1,1);
\draw[thin,opacity=.5] (2,-1.5) to (2,1);
\draw[thin,opacity=.5] (3,-1.5) to (3,1);
\node at (-4,1) {$\circ$};
\node at (-3,1) {$\bullet$};
\node at (-2,1) {$\circ$};
\node at (-2,-1) {$\circ$};
\node at (-1,1) {$\bullet$};
\node at (-0.7,1.2) {\tiny $2$};
\node at (0,1) {$\circ$};
\node at (1,1) {$\bullet$};
\node at (1.3,1.2) {\tiny $2$};
\node at (2,-1) {$\circ$};
\node at (3,1) {$\bullet$};
\node at (-5.5,1) { $a^{1}$};
\node at (-5.5,-1) { $a^{-1}$};
\node at (-4,-2) {$q^{-4}$};
\node at (-2,-2) {$q^{-2}$};
\node at (0,-2) {$q^{0}$};
\node at (2,-2) {$q^{2}$};
\end{tikzpicture} 
    \]
The display here shows inactive variables after active variables as in
\cite{SW}. Since the tangle is of type $UP_{par}$, one can also present the
generating function data in active quiver form. The $1$-colored evaluation gives
a lower bound of $3$ active and $5$ inactive summation indices, which is
achieved by the following data.

\[
    \scalemath{0.65}{\left(\begin{array}{c|c|c}
        1 &1 & 1 \\ \hline 
        1 & 0 & 1 \\ \hline 
        1 & -1 & 1 \\ \hline\hline 
        0 & 0 & 1 \\ \hline 
        0 & -1 & 1 \\ \hline 
        0 & -2 & 1 \\ \hline 
        0 & -1 & -1 \\ \hline 
        0 & -3 & -1 
        \end{array}\right)}
       ,
       \scalemath{0.65}{\left(\begin{array}{c|c|c|c|c|c|c|c} 
        0 & -1 & -2 & 0 & -1 & -2 & 2 & 0 \\ \hline 
        -1 & -1 & -2 & -1 & -1 & -2 & 1 & 0 \\ \hline 
        -2 & -2 & -2 & -1 & -1 & -2 & 1 & 0 \\ \hline 
        0 & -1 & -1 & 0 & -2 & -2 & 1 & -1 \\ \hline 
        -1 & -1 & -1 & -2 & -1 & -2 & -1 & -1 \\ \hline 
        -2 & -2 & -2 & -2 & -2 & -2 & 0 & -1 \\ \hline 
        2 & 1 & 1 & 1 & -1 & 0 & 3 & 1 \\ \hline 
        0 & 0 & 0 & -1 & -1 & -1 & 1 & 1 
        \end{array} \right)}
         \quad
    \begin{tikzpicture} [scale=.4,anchorbase]
    \draw[thin,opacity=.5] (-4.5,-1) to (2,-1);
    \draw[thin,opacity=.5] (-4.5,0) to (2,0);
    \draw[thin,opacity=.5] (-4.5,1) to (2,1);
    \draw[thin,opacity=.5] (-4,-1.5) to (-4,1);
    \draw[thin,opacity=.5] (-3,-1.5) to (-3,1);
    \draw[thin,opacity=.5] (-2,-1.5) to (-2,1);
    \draw[thin,opacity=.5] (-1,-1.5) to (-1,1);
    \draw[thin,opacity=.5] (0,-1.5) to (0,1);
    \draw[thin,opacity=.5] (1,-1.5) to (1,1);
    \draw[thin,opacity=.5] (2,-1.5) to (2,1);
    \node at (-4,1) {$\circ$};
    \node at (-3,1) {$\bullet$};
    \node at (-2,1) {$\circ$};
    \node at (-2,-1) {$\circ$};
    \node at (-1,1) {$\bullet$};
    \node at (0,1) {$\circ$};
    \node at (1,1) {$\bullet$};
    \node at (2,-1) {$\circ$};
    \node at (-5.5,1) { $a^{1}$};
    \node at (-5.5,-1) { $a^{-1}$};
    \node at (-4,-2) {$q^{-4}$};
    \node at (-2,-2) {$q^{-2}$};
    \node at (0,-2) {$q^{0}$};
    \node at (2,-2) {$q^{2}$};
    \end{tikzpicture} 
        \]
The two expressions are related by the identity from Remark~\ref{rem:deactivate}.

\subsection{The (2,-3)-pretzel tangle}
    The following shows the quiver form generating function data for the skein
    module element of the $(2,-3)$-pretzel tangle. The diagram on the right
    shows the $9$ monomials, $6$ active and $3$ inactive, in the $1$-colored
    skein evaluation, which gives a lower bound on the size of any quiver
    describing the generating function. This lower bound is achieved by the
    following data, which was also found experimentally.
    \[
        \begin{tikzpicture} [scale=.5,anchorbase,xscale=-1,yscale=-1]
            \draw[thick] (0,1) \pu (1,2);
            \draw[thick] (0,2) \pu (1,3);
            \draw[thick] (0,3) \pu (1,4);
            \draw[thick] (3,2.5) \pu (2,3.5);
            \draw[thick] (3,1.5) \pu (2,2.5);
            \draw[white, line width=.15cm] (1,1) \pu (0,2);
            \draw[white, line width=.15cm] (1,2) \pu (0,3);
            \draw[white, line width=.15cm] (1,3) \pu (0,4);
            \draw[white, line width=.15cm] (2,2.5) \pu (3,3.5);
            \draw[white, line width=.15cm] (2,1.5) \pu (3,2.5);
            \draw[thick] (1,1) \pu (0,2);
            \draw[thick] (1,2) \pu (0,3);
            \draw[thick] (1,3) \pu (0,4);
            \draw[thick] (2,2.5) \pu (3,3.5);
            \draw[thick] (2,1.5) \pu (3,2.5);
            \draw[thick,<-] (0,0) to (0,1); 
            \draw[thick] (1,4) \ur (1.5,4.5) \rd (2,4) \pd (2,3.5) 
            (2,1.5) \pd (2,1) \dl (1.5,.5) \lu (1,1)
            (0,4) to (0,5);
            \draw[thick,<-] (3,0) to (3,1.5);
            \draw[thick](3,3.5) to (3,5);
            \end{tikzpicture}    
    \qquad  
    \scalemath{0.65}{\left(\begin{array}{c|c|c}
        0& 4 & 1 \\ \hline 
        0& 3 & 1 \\ \hline 
        0&  2 & 1 \\ \hline 
        0&  3 & 1 \\ \hline 
        0&  2 & 1 \\ \hline 
        0&  1 & 1 \\ \hline \hline
        0&  5 & 1 \\ \hline 
        0& 1 & 1 \\ \hline 
        0&  0 & -1 
        \end{array}\right)}
       ,
       \scalemath{0.65}{\left(\begin{array}{c|c|c|c|c|c|c|c|c} 
        -1 & -2 & -2 & -2 & -3 & -3 & -1 & -3 & -1 \\ \hline 
        -2 & -2 & -3 & -2 & -3 & -4 & -1 & -3 & -1 \\ \hline 
        -2 & -3 & -3 & -2 & -3 & -4 & -1 & -4 & -2 \\ \hline 
        -2 & -2 & -2 & -2 & -3 & -3 & -1 & -3 & -1 \\ \hline 
        -3 & -3 & -3 & -3 & -3 & -4 & -1 & -3 & -1 \\ \hline 
        -3 & -4 & -4 & -3 & -4 & -4 & -1 & -4 & -2 \\ \hline 
        -1 & -1 & -1 & -1 & -1 & -1 & -1 & -2 & -1 \\ \hline 
        -3 & -3 & -4 & -3 & -3 & -4 & -2 & -3 & -2 \\ \hline 
        -1 & -1 & -2 & -1 & -1 & -2 & -1 & -2 & 0 
        \end{array} \right)}
         \quad
\begin{tikzpicture} [scale=.4,anchorbase]
\draw[thin,opacity=.5] (-3.5,-1) to (4,-1);
\draw[thin,opacity=.5] (-3.5,0) to (4,0);
\draw[thin,opacity=.5] (-3.5,1) to (4,1);
\draw[thin,opacity=.5] (-3,-1.5) to (-3,1);
\draw[thin,opacity=.5] (-3,-1.5) to (-3,1);
\draw[thin,opacity=.5] (-2,-1.5) to (-2,1);
\draw[thin,opacity=.5] (-1,-1.5) to (-1,1);
\draw[thin,opacity=.5] (0,-1.5) to (0,1);
\draw[thin,opacity=.5] (1,-1.5) to (1,1);
\draw[thin,opacity=.5] (2,-1.5) to (2,1);
\draw[thin,opacity=.5] (3,-1.5) to (3,1);
\draw[thin,opacity=.5] (4,-1.5) to (4,1);
\node at (-3,1) {$\bullet$};
\node at (-2,1) {$\circ$};
\node at (-1,1) {$\bullet$};
\node at (-0.7,1.2) {\tiny $2$};
\node at (0,-1) {$\circ$};
\node at (1,1) {$\bullet$};
\node at (1.3,1.2) {\tiny $2$};
\node at (3,1) {$\bullet$};
\node at (4,1) {$\circ$};
\node at (-4.5,1) { $a^{1}$};
\node at (-4.5,-1) { $a^{-1}$};
\node at (-2,-2) {$q^{-2}$};
\node at (0,-2) {$q^{0}$};
\node at (2,-2) {$q^{2}$};
\node at (4,-2) {$q^{4}$};
\end{tikzpicture} 
        \]
Just like before, this tangle is of type $UP_{par}$ and admits a generating
        function in active quiver form. The $1$-colored evaluation gives a lower
        bound of $3$ active and $3$ inactive generators for such an expression, which is achieved by the following data.

        \[
            \scalemath{0.65}{\left(\begin{array}{c|c|c}
                1&  3 & 1 \\ \hline 
                1& 2 & 1 \\ \hline 
                1&  1 & 1 \\ \hline \hline
                0&  5 & 1 \\ \hline 
                0&  1 & 1 \\ \hline 
                0&  0 & -1 
                \end{array}\right)}
               ,
               \scalemath{0.65}{\left(\begin{array}{c|c|c|c|c|c} 
                -2 & -3 & -3 & -1 & -3 & -1 \\ \hline 
                -3 & -3 & -4 & -1 & -3 & -1 \\ \hline 
                -3 & -4 & -4 & -1 & -4 & -2 \\ \hline 
                -1 & -1 & -1 & -1 & -2 & -1 \\ \hline 
                -3 & -3 & -4 & -2 & -3 & -2 \\ \hline 
                -1 & -1 & -2 & -1 & -2 & 0 
                \end{array} \right)}
                \quad
\begin{tikzpicture} [scale=.4,anchorbase]
\draw[thin,opacity=.5] (-3.5,-1) to (4,-1);
\draw[thin,opacity=.5] (-3.5,0) to (4,0);
\draw[thin,opacity=.5] (-3.5,1) to (4,1);
\draw[thin,opacity=.5] (-3,-1.5) to (-3,1);
\draw[thin,opacity=.5] (-3,-1.5) to (-3,1);
\draw[thin,opacity=.5] (-2,-1.5) to (-2,1);
\draw[thin,opacity=.5] (-1,-1.5) to (-1,1);
\draw[thin,opacity=.5] (0,-1.5) to (0,1);
\draw[thin,opacity=.5] (1,-1.5) to (1,1);
\draw[thin,opacity=.5] (2,-1.5) to (2,1);
\draw[thin,opacity=.5] (3,-1.5) to (3,1);
\draw[thin,opacity=.5] (4,-1.5) to (4,1);
\node at (-3,1) {$\bullet$};
\node at (-2,1) {$\circ$};
\node at (-1,1) {$\bullet$};
\node at (0,-1) {$\circ$};
\node at (1,1) {$\bullet$};
\node at (4,1) {$\circ$};
\node at (-4.5,1) { $a^{1}$};
\node at (-4.5,-1) { $a^{-1}$};
\node at (-2,-2) {$q^{-2}$};
\node at (0,-2) {$q^{0}$};
\node at (2,-2) {$q^{2}$};
\node at (4,-2) {$q^{4}$};
\end{tikzpicture} 
                \]
                The two expressions are again related by the identity from Remark~\ref{rem:deactivate}.

                    
            

                \Addresses
                
                \end{document}